\documentclass[a4paper,12pt]{amsart}
\usepackage{amsmath}                    
\usepackage{amscd,amsthm,a4wide}
\usepackage{amsfonts}                   
\usepackage{amssymb}                    
\usepackage[all]{xy}                    
\usepackage[latin1]{inputenc}           
\usepackage[bookmarks=true,colorlinks=true,naturalnames=true]{hyperref} 
\usepackage{pdfsync}                    
\usepackage{times}                      

\newtheorem{theorem}                 {Theorem}      [section]
\newtheorem{proposition}  [theorem]  {Proposition}
\newtheorem{corollary}    [theorem]  {Corollary}
\newtheorem{lemma}        [theorem]  {Lemma}

\theoremstyle{definition}
\newtheorem{example}      [theorem]  {Example}
\newtheorem{remark}       [theorem]  {Remark}
\newtheorem{definition}   [theorem]  {Definition}

\numberwithin{equation}{section}

\let\oldmarginpar\marginpar
\renewcommand\marginpar[1]{\ \oldmarginpar[\raggedleft\footnotesize #1]{\raggedright\footnotesize #1}}

\renewcommand\arraycolsep{1pt}

\def \U{\text{\bf U}}
\def \u{\mathfrak{u}}
\def \gl{\mathfrak{gl}}

\def \fieldC{\mathbb{C}}

\def \d{\mathrm{d}}
\def \wordspan{\mathrm{span}}
\def \wordrank{\mathrm{rank}}
\def \image{\mathrm{Im}}
\def \End{\mathrm{End}}

\def \wordalg{\mathrm{alg}}

\def \H{\mathcal{H}}

\begin{document}

\title[Explicit construction of harmonic two-spheres into the complex Grassmannian]{Explicit construction of harmonic two-spheres into the complex Grassmannian}

\author{Maria João Ferreira}
\author{Bruno Ascenso Simões}

\address{\textit{MJF}, \textit{BAS}: Centro de Matemática e Aplicações Fundamentais, Universidade de Lisboa, Av.\ Prof.\ Gama Pinto 2\\1649-003, Lisbon, Portugal}

\address{\textit{MJF}: Departamento de Matemática da FCUL, C6, 2º piso, Campo Grande, 1749-016 Lisboa, Portugal}

\address{\textit{BAS}: Universidade Lusófona, Núcleo de Investigação em Matemática, Campo Grande, 376, 1749-024 Lisbon, Portugal}

\email{mjferr@ptmat.fc.ul.pt; b.simoes@ptmat.fc.ul.pt}

%

\begin{abstract}
We present an explicit description of all harmonic maps of finite uniton number from a Riemann surface into a complex Grassmannian. Namely, starting from a constant map $Q$ and a collection of meromorphic functions and their derivatives, we show how to algebraically construct all harmonic maps from the two-sphere into a given Grassmannian $G_p(\mathbb C^n)$. In this setting the uniton number depends on $Q$ and $p$ and we obtain a sharp estimate for it.
%
%
%
\end{abstract}


\maketitle

\thispagestyle{empty}

\section*{Introduction}

Harmonic spheres in complex Grassmannians have been extensively studied using various techniques (see \cite{BurstallWood:86,ChernWolfson:85,ChernWolfson:87}). As it is well-known, the complex Grassmannian sits naturally in the unitary group $\U(N)$ equipped with its standard bi-invariant metric, via its Cartan totally geodesic embedding. Using a Bäcklund transformation technique, Uhlenbeck \cite{Uhlenbeck:89} obtained a method to construct successive harmonic maps into $\U(N)$ from an initial harmonic map. She proved that through this process, called ``adding a uniton'', one can obtain all harmonic maps from a Riemann surface with finite uniton number. Subsequent works have expanded this view. However obtaining explicit unitons involves solving $\overline{\partial}$-problems which is a difficult task \cite{HeShen:04,Wood:89}. In \cite{FerreiraSimoesWood:09} J.~C. Wood and the authors gave an algebraic procedure to construct these unitons so that one can build all harmonic maps with finite uniton number from a Riemann surface into $\U(N)$, from freely chosen meromorphic functions into $\mathbb C^n$ and their derivatives. Although these harmonic maps include those with values in the Grassmannian, no explicit way was given to decide when, from a specific meromorphic data, one could obtain a Grassmannian-valued harmonic map. The aim of this paper is to study, from this point of view, harmonic maps with finite uniton number from a Riemann surface into a Grassmannian manifold. More specifically, we present algebraic conditions, to be satisfied by the initial data, ensuring that the obtained harmonic maps have values in a Grassmannian (Theorem \ref{Theorem:MainTheoremNonBundleVersion}). Furthermore, for a specific Grassmannian manifold $G_p(\mathbb C^n)$, we show how to organize our initial data so that the harmonic map has its image in the given Grassmannian manifold (Theorem \ref{Theorem:SynthesizedVersion}).

Associated to a harmonic map $\phi: M^2 \rightarrow \U(n)$, there is a spectral deformation, called the extended solution; that is a family of maps $\Phi_{\lambda}:M^2 \rightarrow \U(n)$, depending smoothly on $\lambda \in S^1$, such that $\phi=Q\Phi_{-1}$ (for some $Q\in \U(n)$) and the differential form $A^{\lambda}=\frac{1}{2}\Phi_{\lambda}^{-1}d \Phi_{\lambda}$ satisfies \cite{Uhlenbeck:89}
\[A^{\lambda}=\frac{1}{2}(1-\lambda^{-1})A_z^{\lambda} + \frac{1}{2}(1-\lambda)A_{\bar{z}}^{\lambda}.\]
The extended solution is not, in general, unique. However, Uhlenbeck proved that, given a harmonic map, there exists a unique extended solution $\Phi_{\lambda}$ of type-one, i.e., such that the image of $\Phi_0$ is full. Furthermore, given a harmonic map $\phi: M^2 \rightarrow G_p(\mathbb C^n)$ with finite uniton number, there exists $Q=\pi_{F_0}-\pi_{F_0}^{\perp}\in \U(n)$, such that $\phi=Q\Phi_{-1}$, where $\Phi_{\lambda}$ denotes the type-one extended solution, and $\pi_{F_0}$ denotes the orthogonal projection onto a complex subspace $F_0$ of $\mathbb C^n$. Under these conditions, we present an estimate for the uniton number of such a harmonic map, depending on $p$ and $Q$. This estimate is sharp. It is known that, for a harmonic map $\phi: M^2 \rightarrow G_p(\mathbb C^n)$, the maximal uniton number is less or equal than $2\min\{p,n-p\}$ (\cite{BurstallGuest:97,DongShen:96}). We show that this value is only attained when $Q=\pm I$. Unlike the case of harmonic maps $\phi:S^2 \rightarrow \U(n)$, for $G_p(\mathbb C^n)$-valued harmonic maps, the possible uniton numbers depend on $Q$.

In \cite{Segal:89}, G.~Segal gave a model for the loop group of $\U(n)$ as an (infinite-dimensional) Grassmannian and showed that harmonic maps of finite uniton number correspond to holomorphic maps into a related finite-dimensional Grassmannian. We interpret our results in the framework of the Grassmannian model and relate them with those in \cite{FerreiraSimoesWood:09}.

The paper is organized as follows: in Section \ref{Section:ExplicitFormulaeForHarmonicMapsIntoUn} we recall Uhlenbeck's factorization an explain the algebraic procedure, presented in \cite{FerreiraSimoesWood:09}, to construct explicit unitons. Section \ref{Section:HarmonicMapsFromHolomorphicData} is devoted to the study of Grassmannian-valued harmonic maps. In \ref{Subsection:ExplicitConstruction} we describe the main results for harmonic maps $\phi:S^2\rightarrow G_*(\mathbb C^n)$ and present examples. Harmonic maps with values in a specific Grassmannian manifold are treated in \ref{Subsection:HarmonicMapsIntoGpCn}. Subsection \ref{Subsection:GrassmannianModel} is devoted to the interpretation of our  construction in the Grassmannian model setting. All involved calculations and proofs are presented, separately, in Subsection \ref{Subsection:Proofs}.

\bigskip

The authors are grateful to John~C. Wood for very useful discussions on this work.


\section{Preliminaries: harmonic maps into $\U(n)$}\label{Section:ExplicitFormulaeForHarmonicMapsIntoUn}


Let $M^2$ be a Riemann surface. For any smooth map $\phi:M^2\to\U(n)$, let $A^{\phi}$ denote one half the pull-back of the Maurer--Cartan form,
\begin{equation}\label{Equation:DefinitionOfA}
A^{\phi}=\textstyle{\frac{1}{2}}\phi^{-1}\d\phi.
\end{equation}

Choosing a local complex coordinate $z$ on an open subset of $M^2$, we write $A^{\phi}=A^{\phi}_z\d z+A^{\phi}_{\bar{z}}\d\bar{z}$, where $A^{\phi}_z$ and $A^{\phi}_{\bar{z}}$ denote the $(1,0)$- and $(0,1)$-parts (with respect to $M^2$), respectively. Let $\underline{\fieldC}^n$ denote the trivial complex bundle $M^2 \times \fieldC^n$ equipped with the standard Hermitian inner product: $<u,v>=u_1\overline{v}_1+\cdots+u_n\overline{v}_n$ ($u=(u_1,\ldots,u_n)$, $v=(v_1,\ldots,v_n)\in\fieldC^n$) on each fibre. $A^{\phi}_z$ and  $A^{\phi}_{\bar{z}}$ are local sections of the endomorphism bundle $\End(\underline{\fieldC}^n)$, and each is minus the adjoint of the other. $D^{\phi}:=\d+A^{\phi}$ is a unitary connection on the trivial bundle $\underline{\fieldC}^n$; in fact, it is the pull-back of the Levi-Civita connection $\U(n)$.

We write $D^{\phi}_z =\partial_z+A^{\phi}_z$ and $D^{\phi}_{\bar{z}}=\partial_{\bar{z}}+A^{\phi}_{\bar{z}}$ where $\partial_z=\partial/\partial z$ and $\partial_{\bar{z}}=\partial/\partial\bar{z}$ for a (local) complex coordinate $z$ on $M^2$. Give $\underline{\fieldC}^n$ the Koszul--Malgrange complex structure \cite{KoszulMalgrange:58}; this is the unique holomorphic structure such that a (local) section $\sigma$ of $\underline{\fieldC}^n$ is holomorphic if and only if $D^{\phi}_{\bar{z}}\sigma=0$ for any complex coordinate $z$; we shall denote the resulting holomorphic bundle by $(\underline{\fieldC}^n,D^{\phi}_{\bar{z}})$. Note that, when $\phi$ is constant, $A^{\phi} = 0$, and the Koszul-Malgrange holomorphic structure is the standard holomorphic structure on $\underline{\fieldC}^n$.

Since $A^{\phi}_z$ represents the derivative $\partial\phi/\partial z$, \emph{the map $\phi$ is harmonic if and only if the endomorphism $A^{\phi}_z$ is holomorphic}, i.e.,
\[A^{\phi}_z\circ D^{\phi}_{\bar{z}}=D^{\phi}_{\bar{z}}\circ A^{\phi}_z\,.\]
Let $\phi:M^2\to\U(n)$ be harmonic and let $\underline{\alpha}$ be a smooth subbundle of $\underline{\fieldC}^n$. We shall say that $\underline{\alpha}$ is \emph{proper} if it is neither the zero subbundle nor the full bundle $\underline{\fieldC}^n$ and we consider that $\underline{\alpha}$ is \emph{full} if is not contained in any proper subspace of $\fieldC^n$. Finally, by a \emph{uniton} or \emph{flag factor for $\phi$} we mean a smooth subbundle $\underline{\alpha}$ such that
\begin{equation}\label{Equation:UnitonConditions}
\left\{
\begin{array}{rrl}
{\rm (i)} & D^{\phi}_{\bar{z}}(\sigma) \in \Gamma(\underline{\alpha})
	&\mbox{ for all } \sigma \in \Gamma(\underline{\alpha})\,, \\
{\rm (ii)} & A^{\phi}_z(\sigma) \in \Gamma(\underline{\alpha})
	&\mbox{ for all } \sigma \in \Gamma(\underline{\alpha})\,;
\end{array} \right.
\end{equation}
here $\Gamma(\cdot)$ denotes the space of smooth sections of a bundle. These equations say that $\underline{\alpha}$ is a holomorphic subbundle of $(\underline{\fieldC}^n, D^{\phi}_{\bar{z}})$ which is closed under the endomorphism $A^{\phi}_z$.

For a subbundle $\underline{\alpha}$ of $\underline{\fieldC}^n$, let $\pi_{\alpha}$ and $\pi_{\alpha}^{\perp}$ denote orthogonal projection onto $\underline{\alpha}$ and onto its orthogonal complement $\underline{\alpha}^{\perp}$, respectively. Then \cite{Uhlenbeck:89}

\begin{proposition}
The map $\widetilde{\phi}:M^2 \to \U(n)$ given by $\widetilde{\phi}=\phi(\pi_{\alpha}-\pi_{\alpha}^{\perp})$ is harmonic if and only if $\underline{\alpha}$ is a uniton.
\end{proposition}

Note that $\underline{\alpha}$ is a uniton for $\phi$ if and only if $\underline{\alpha}^{\perp}$ is a uniton for $\widetilde{\phi}$; further
$\phi=-\widetilde{\phi}(\pi_{\alpha}^{\perp}-\pi_{\alpha})$\; i.e., the flag transforms defined by $\underline{\alpha}$ and $\underline{\alpha}^{\perp}$ are inverse up to sign.

Given a harmonic map $\phi$ and a uniton $\underline{\alpha}$ for $\phi$, we can characterize the holomorphic structure $D_{\bar{z}}^{\widetilde{\phi}}$ as well as the operator $A_z^{\widetilde{\phi}}$ for the new harmonic map $\widetilde{\phi}=\phi(\pi_{\alpha}-\pi_{\alpha}^\perp)$ by the simple formulae \cite{Uhlenbeck:89}
\begin{equation*}\label{uniton-DA}
A_z^{\widetilde{\phi}} = A_z^{\phi} + \partial_z\pi_{\alpha}^{\perp}\,,\quad
D_{\bar{z}}^{\widetilde{\phi}} = D_{\bar{z}}^{\phi} - \partial_{\bar{z}}\pi_{\alpha}^{\perp}.
\end{equation*}
Hence, we can also write down the uniton equations \eqref{Equation:UnitonConditions} for the harmonic map $\widetilde{\phi}$. In general, finding unitons for the harmonic map $\widetilde{\phi}$ would require to solve a $\overline{\partial}$-problem. However, the following result (\cite{FerreiraSimoesWood:09}, Theorem 1.1.) gives an explicit construction of these unitons (for a different approach, see \cite{DaiTerng:07}).

\begin{theorem}\label{Theorem:Theorem11FromFerSimWood09} For any $r\in\{0,1,\ldots, n-1\}$, let $(H_{i,j})_{0 \leq i \leq r-1,1\leq j\leq n}$ be an $r\times n$ array of $\fieldC^n$-valued meromorphic functions on $M^2$, and let $\phi_0$ be an element of $\U(n)$.
For each $i = 0,1,\ldots,r-1$, set $\underline{\alpha}_{i+1}$ equal to the subbundle of $\underline{\fieldC}^n$ spanned by the vectors
\begin{equation}\label{Equation:FirstEquationInTheorem:Theorem11FromFerSimWood09}
\alpha^{(k)}_{i+1,j}=\sum_{s=k}^{i} C^{i}_s H^{(k)}_{s-k, j}\qquad(j = 1,\ldots, n, \ k =0,1, \ldots, i).
\end{equation}
Then, the map  $\phi:M^2 \to \U(n)$ defined by
\begin{equation}\label{Equation:SecondEquationInTheorem:Theorem11FromFerSimWood09}
\phi=\phi_0 (\pi_1-\pi_1^{\perp})\cdots(\pi_r-\pi_r^{\perp})
\end{equation}
is harmonic.

Further, all harmonic maps of finite uniton number, and so all harmonic maps from $S^2$, are obtained this way.
\end{theorem}

In the above result, by a \emph{$\fieldC^n$-valued meromorphic function} or \emph{meromorphic vector} $H$ on $M^2$, it is simply meant an $n$-tuple of meromorphic functions; its $k$'th derivative with respect to some local complex coordinate on $M^2$ is denoted by $H^{(k)}$. Also, $\pi_i$ denotes $\pi_{\underline{\alpha}_i}$ whereas $\pi_i^{\perp}$ stands for $\pi_{\underline{\alpha}_i^\perp}$. Moreover, for integers $i$ and $s$ with $0\leq s\leq i$, $C^i_s$ denotes the \emph{$s$'th elementary function} of the projections $\pi_i^{\perp}, \ldots, \pi_1^{\perp}$ given by
\begin{equation}\label{Equation:DefinitionOfCis}
C^i_s=\sum_{1\leq i_1<\cdots<i_s\leq i}\pi_{i_s}^{\perp}\cdots\pi_{i_1}^{\perp}.
\end{equation}
$C^i_s$ denotes the identity when $s = 0$ and zero when $s<0$ or $s>i$. Note that the $C^i_s$ satisfy a property like that for Pascal's triangle
\begin{equation}\label{Equation:PascalIdentityForCis}
C^i_s=\pi_i^{\perp} C^{i-1}_{s-1}+ C^{i-1}_s\quad(i\geq 1,0\leq s\leq i).
\end{equation}
Moreover, the unitons $\underline{\alpha}_i$ satisfy the \emph{covering condition}
\begin{equation}\label{Equation:CoverCondition}
\pi_{i}\underline{\alpha}_{i+1}=\underline{\alpha}_{i}  \qquad (i = 1,\ldots,r-1).
\end{equation}

We quickly review the main steps in the proof of the above theorem (for more details, we refer the reader to \cite{FerreiraSimoesWood:09}). To see that the map $\phi$ in \eqref{Equation:SecondEquationInTheorem:Theorem11FromFerSimWood09} is harmonic, all one has to do is to check that the successive bundles $\underline{\alpha}_{i+1}$ satisfy equations \eqref{Equation:UnitonConditions} for each of the harmonic maps $\phi_i$, where
\[\phi_i=\phi_0 (\pi_1-\pi_1^{\perp})\cdots(\pi_i-\pi_i^{\perp}).\]
This follows from an explicit calculation showing that (\cite{FerreiraSimoesWood:09}, Proposition 2.4.):

(i) $\alpha_{i+1,j}^{(k)}$ are holomorphic sections of\/ $(\underline{\fieldC}^n, D^{\phi_i}_{\bar{z}})$ and

(ii) $A^{\phi_i}_z(\alpha_{i+1,j}^{(k)})
	= \left\{\begin{array}{cl}
		-\alpha_{i+1,j}^{(k+1)}\,, & \mbox{ if \ $k < i+1$\,$,$} \\
			0\,, 		&  \mbox{ if \ $k=i+1$\,$;$}
	\end{array}\right.$

As for the converse, one needs to develop further the theory. Let $A^\phi$ be as in \eqref{Equation:DefinitionOfA} and set
\[A^{\lambda}=\textstyle{\frac{1}{2}}(1-\lambda^{-1})A^\phi_z\d z+\textstyle{\frac{1}{2}}(1-\lambda)A^\phi_{\bar{z}}\d\bar{z}\qquad(\lambda\in S^1).\]
It is well-known that the harmonicity of $\phi$ implies the integrability of $A^\lambda$ and we can therefore find, at least locally, an $S^1$-family of smooth maps $\Phi=\Phi_{\lambda}:M^2 \to \U(n)$ with
\[\textstyle{\frac{1}{2}}\Phi_{\lambda}^{-1}\d\Phi_{\lambda}=A^{\lambda}\qquad(\lambda\in S^1)\qquad\text{and}\qquad\Phi_1(z)=I\text{ for all }z\in M^2,\]
where $I$ is the identity matrix. We say that $\Phi=\Phi_{\lambda}:M^2\to\U(n)$ is an \emph{extended solution} \cite{Uhlenbeck:89} (for $\phi$) and it is clear that $\Phi$ can be interpreted as a map into a loop group.

Note that any two extended solutions for a harmonic map differ by a function (`constant loop') $Q:S^1\to\U(n)$ with $Q(1)=1$. Further, $\Phi_{-1}$ is left-equivalent to $\phi$, i.e., $\Phi_{-1}=Q\phi$ for some constant $Q\in\U(n)$.

Let $\gl(n,\fieldC)$ denote the Lie algebra of $n \times n$ matrices; this is the complexification of $\u(n)$. The extended solution extends to a family of maps $\Phi_{\lambda}:M^2 \to \gl(n,\fieldC)$ with $\Phi_{\lambda}$ a holomorphic function of $\lambda\in\fieldC\setminus\{0\}$. Hence it can be expanded as a Laurent series, $\Phi=\sum_{i=-\infty}^{\infty}\lambda^iT_i$, where each $T_i=T_i^{\Phi}$ is a smooth map from $M^2$ to $\gl(n,\fieldC)$.

A harmonic map $\phi:M^2\to\U(n)$ is said to be \emph{of finite uniton number} if it has a \emph{polynomial} extended solution
\begin{equation}\label{Equation:PolinomialExtendedSolution}
\Phi=T_0+\lambda T_1+\cdots+\lambda^r T_r.
\end{equation}
The \emph{(minimal) uniton number} of $\phi$ is the least degree of all its polynomial extended solutions. In general, given a harmonic map $\phi$ (with finite uniton number $r$) there is not a unique corresponding extended solution $\Phi$ with degree $r$. Nevertheless, if one further imposes that the subbundle $\underline{\image}T_0$ is full, uniqueness is achieved \cite{Uhlenbeck:89}. Such extended solutions have a unique factorization
\begin{equation}\label{Equation:FactorizationOfTheHarmonicMapIntoUnitons}
\Phi=(\pi_1+\lambda\pi_1^{\perp}) \cdots (\pi_r+\lambda\pi_r^{\perp})
\end{equation}
where $\underline{\alpha}_1,...,\underline{\alpha}_r$ are proper unitons satisfying the covering condition \eqref{Equation:CoverCondition} and $\underline{\alpha}_1$ is full; these will be called \emph{type-one} extended solutions. One can then prove that each of these subbundles $\underline{\alpha}_i$ is of the form stated in Theorem \ref{Theorem:Theorem11FromFerSimWood09}.

\begin{example}\label{Example:MapsIntoU3}\cite{FerreiraSimoesWood:09}
Let $\phi:M^2 \to \U(3)$ be a non-constant harmonic map of finite uniton number. Then, \emph{either}

(a) it has uniton number one and is given by a holomorphic map $\phi:M^2\to G_{d_1}(\fieldC^3)$ where $d_1=1$ or $2$; \emph{or}

(b) it has uniton number two and is given by \eqref{Equation:FirstEquationInTheorem:Theorem11FromFerSimWood09} with unitons
$\underline{\alpha}_1$, $\underline{\alpha}_2$ of rank one and two respectively and $\underline{\alpha}_1$ full.
The data of Theorem \ref{Theorem:Theorem11FromFerSimWood09} consists of maps $H_{0,1}$ and $H_{1,1}$. Then, since $A_z^{\phi_1}(H_{0,1})=-\pi_1^\perp H_{0,1}^{(1)}$,
\begin{equation}\label{Equation:Rank12}
\underline{\alpha}_1=\wordspan\{H_{0,1}\}\quad\text{and}\quad\underline{\alpha}_2=\wordspan\{H_{0,1}+\pi_1^{\perp}H_{1,1},\pi_1^{\perp}H_{0,1}^{(1)}\}.
\end{equation}
\end{example}

\section{Harmonic maps into $G_*(\fieldC^n)$}\label{Section:HarmonicMapsFromHolomorphicData}



\subsection{Explicit construction}\label{Subsection:ExplicitConstruction}


$\text{ }$

For any $p\in\{0,1,\ldots, n\}$, let $G_p(\mathbb C^n)$ denote the \emph{complex Grassmannian} of $k$-dimensional subspaces of $\mathbb C^n$ equipped with its standard structure as a Hermitian symmetric space. It is convenient to denote the disjoint union $\cup_{p=0}^n G_p(\mathbb C^n)$ by $G_*(\mathbb C^n)$. In the sequel, we always identify a map into $G_p(\mathbb C^n)$ with the pull-back of the corresponding tautological bundle. As it is well-known, $G_p(\mathbb C^n)$ sits totally geodesically in $\U(n)$ via the \emph{Cartan embedding} $\iota(F)=\pi_F-\pi_F^\perp$.
The formulae in Theorem \ref{Theorem:Theorem11FromFerSimWood09} gives all harmonic maps from $S^2$ into $G_*(\mathbb C^n)$, although it does not tell us \emph{how to choose the holomorphic data $H_{i,j}$} in order to guarantee that the resulting map $\phi$ lies in $G_*(\mathbb C^n)$. On the other hand the situation is now somehow different, in the sense that left multiplication by a constant map $Q$ does not, in general, preserve the image in $G_*(\mathbb C^n)$ \cite{Uhlenbeck:89}. Therefore the classification up to left multiplication is no longer suitable in this setting.

\begin{example}\label{Example:MapsIntoGAst3}
Let $\phi:M^2\to G_*(\mathbb C^n)$ be a non-constant harmonic map of uniton number one. Then, if $\phi$ is not holomorphic, it must be of the form

$\phi=(\pi_{F_0}-\pi_{F_0}^\perp)(\pi_1-\pi_1^\perp)$.

It is easily seen that $\phi$ is $G_*(\mathbb C^n)$-valued if, and only if, $\pi_{F_0}$ and $\pi_1$ commute; equivalently, $F_0$ \emph{decomposes} $\underline{\alpha}_1$; i.e.,
\[\underline{\alpha}_1= \underline{\alpha}_1\cap F_0\oplus \underline{\alpha}_1 \cap F_0^\perp.\]
In that case, we can easily check that
\[\phi=\pi_{F_1}-\pi_{F_1}^\perp\]
where
\begin{equation}\label{Equation:DefinitionOfF1}
\underline{F}_1=\underline{\alpha}_1\cap F_0\oplus\underline{\alpha}_1^\perp\cap F_0^\perp.
\end{equation}
Notice that if $F_0$ is not trivial and $\underline{\alpha}_1$ is full, then it must be that rank $\underline{\alpha}_1\geq2$. Moreover, in that case, $\phi_1$ decomposes into $\phi_1\cap\underline{\alpha}_1$ and $\phi_1\cap\underline{\alpha}_1^\perp$ which are, respectively, holomorphic and anti-holomorphic subbundles of $(\underline{\mathbb C}^n,\partial_{\bar{z}})$.
\end{example}

\begin{example}
A harmonic map $\phi: M^2 \rightarrow G_*(\mathbb C^n)$ with uniton number $2$ can be written as $\phi=(\pi_{F_0}-\pi_{F_0^\perp})(\pi_1-\pi_1^{\perp})(\pi_2-\pi_2^\perp)$, where $F_0$ is a complex subspace of $\mathbb C^n$ and $\underline{\alpha}_1$ is full. From (\cite{Uhlenbeck:89}, Theorem 15.3) we know that $\phi_1=(\pi_{F_0}-\pi_{F_0^{\perp}})(\pi_1-\pi_1^\perp)$ must be also Grassmannian-valued. As in Example \ref{Example:MapsIntoGAst3}, $F_0$ splits $\underline{\alpha}_1$ and $\phi_1=\pi_{F_1}-\pi_{F_1^\perp}$, where $\underline{F}_1$ is given by \eqref{Equation:DefinitionOfF1}. Again, since $\phi$ has values in $G_*(\mathbb C^n)$, $\pi_2$ and $\pi_{F_1}$ commute which implies that $\underline{F}_1$ splits $\underline{\alpha}_2$ and $\phi=\pi_{F_2}-\pi_{F_2}^{\perp}$, where
\begin{equation}\label{Equation:DefinitionOfF2}
\underline{F}_2= \underline{\alpha}_2 \cap \underline{F}_1 \oplus \underline{\alpha}_2^{\perp} \cap \underline{F}_1^{\perp}.
\end{equation}

When $F_0$ is trivial (i.e. $F_0 = \mathbb C^n$ or $F_0 = \{0\}$), it is easily seen, from the covering condition and the fact that $\pi_1$ and $\pi_2$ commute, that $\underline{\alpha}_1 \subset \underline{\alpha}_2$. Hence, according to Theorem \ref{Theorem:Theorem11FromFerSimWood09},
\[\begin{array}{l}
\underline{\alpha}_1=\wordspan\{H_{0,1},...,H_{0,r}\}\\
\underline{\alpha}_2=\wordspan\{H_{0,1},..., H_{0,r},\pi_1^\perp H_{0,1}^{(1)},...,\pi_1^\perp H_{0,r}^{(1)}\},
\end{array}\]
for some meromorphic data $H_{0,1},...,H_{0,r}$.

Assume now that $F_0$ is not trivial and choose meromorphic data $\{L_{0,i}\}_{1\leq i\leq r}$ in $F_0$ and
$\{E_{0,j}\}_{1\leq j\leq l}$ in $F_0^{\perp}$ to span $\underline{\alpha}_1$. From Theorem \ref{Theorem:Theorem11FromFerSimWood09} we know that
\[\begin{array}{l}
\underline{\alpha}_1 =\wordspan\{L_{0,i},E_{0,j}\}_{(1\leq i\leq r,~1\leq j\leq l)}\\
\underline{\alpha}_2 =\wordspan \{L_{0,i}+\pi_1^\perp H_{1,i}, E_{0,j}+\pi_1^\perp H_{1,r+j}, \pi_1^\perp L_{0,i}^{(1)},\pi_1^\perp E_{0,j}^{(1)}\}_{1\leq i\leq r,~1\leq j\leq l},
\end{array}\]
where the $\{H_{1,s}\}_{1\leq s\leq r+l}$ are $\mathbb C^n$-valued meromorphic functions.

It is easily seen that if the $H_{1,i}$ ($1\leq i\leq r$) lie in $F_0^\perp$ and the $H_{1,r+j}$ ($1\leq j \leq l$) lie in $F_0$ then $\phi_1=(\pi_{F_0}-\pi_{F_0^\perp})(\pi_1-\pi_1^\perp)$ commutes with $\pi_2$. As we shall see, eventually rearranging indexes, $\alpha_2$ must be given this way.

Notice that, in the decomposition of $\underline{F}_2$ given by \eqref{Equation:DefinitionOfF2}, $\underline{\alpha}_2 \cap F_1$ is a holomorphic subbundle of $(\mathbb C^n, D_{\overline{z}}^{\phi_1})$, since it is spanned by the sections $\{L_{0,i}+\pi_1^\perp H_{1,i}\}$ and $\{E_{0,j}+\pi_1^\perp H_{1,r+j}\}$ ($1\leq i\leq r, ~1\leq j\leq l$), which are holomorphic sections of that bundle. As we shall see later on, $\underline{\alpha}_2^{\perp} \cap F_1^{\perp}$ is a anti-holomorphic subbundle of the same bundle.
\end{example}

One of the main ingredients to develop the theory when dealing with harmonic maps into $G_*(\mathbb C^n)$ is the following result, already suggested by the previous examples.

\begin{proposition}\label{Proposition:DecompositionIntoHoloAndAntiHoloPartsOfHarmonicMapsIntoGStarCn} Let $\phi:M^2\to G_*(\mathbb C^n)$ be a harmonic map and $\underline{\alpha}$ a uniton for $\phi$. Then, the harmonic map $\tilde{\phi}=\phi(\pi_{\alpha}-\pi_{\alpha}^\perp)$ is $G_*(\mathbb C^n)$-valued if, and only if, $\phi$ splits $\underline{\alpha}$. In that case, $\tilde{\phi}=\phi\cap\underline{\alpha}\oplus\phi^\perp\cap\underline{\alpha}^\perp$, where $\phi\cap\underline{\alpha}$ and $\phi^\perp\cap\underline{\alpha}^\perp$ are, respectively, holomorphic and anti-holomorphic subbundles of $(\mathbb C^n,D_{\bar{z}}^\phi)$.
\end{proposition}

In the case of harmonic maps $\phi:M^2 \rightarrow \U(n)$, the holomorphic data $H_{i,j}$ of Theorem \ref{Theorem:Theorem11FromFerSimWood09} could be freely chosen. We may inquire which conditions we must impose to the $H_{i,j}$ to get $\phi(M^2)\subseteq G_k(\mathbb C^n)$. The preceding proposition indicates that the splitting idea must be present in the initial data in order to obtain Grassmannian-valued harmonic maps.

\begin{definition}
Let $F_0$ be a constant subspace in $\mathbb C^n$. An $r\times n$ \emph{$F_0$-array} is a family of meromorphic $\mathbb C^n$-valued functions, $(K_{i,j})_{0\leq i\leq r-1, 1\leq j\leq n}$ such that, for each $j$, either
\begin{equation}
\begin{array}{l}
\pi_{F_0^\perp}(K_{2k,j})=0\text{ and }\pi_{F_0}(K_{2k+1,j})=0\text{, for all $0\leq k\leq\frac{r-1}{2}$ or}\\
\pi_{F_0}(K_{2k,j})=0\text{ and }\pi_{F_0^\perp}(K_{2k+1,j})=0\text{, for all $0\leq k\leq\frac{r-1}{2}$.}
\end{array}
\end{equation}
\end{definition}

\begin{theorem}\label{Theorem:MainTheoremNonBundleVersion}
Let $F_0$ be a constant subspace in $\mathbb C^n$, $r\in\{0,1,...,n-1\}$ and $(K_{i,j})_{0\leq i\leq r-1,1\leq j\leq n}$ be an $r\times n$ $F_0$-array of $\mathbb C^n$-valued meromorphic functions on $M^2$. For each $j$, consider the meromorphic functions
\begin{equation}\label{Equation:FirstEquationInTheorem:MainTheoremNonBundleVersion}
\begin{array}{l}
H_{0,j}=K_{0,j}\text{ and}\\
H_{i,j}=\displaystyle\sum_{s=1}^i(-1)^{s+i}\binom{i-1}{s-1}K_{s,j}\text{,  }i\geq 1.
\end{array}
\end{equation}
For each $0\leq i\leq r-1$, set $\underline{\alpha}_{i+1}$ equal to the subbundle of $\mathbb C^n$ spanned by the vectors
\[\alpha_{i+1,j}^{(k)}=\sum_{s=k}^i C^i_s H_{s-k,j}^{(k)},\,(j=1,...,n,\,k=0,...,i).\]
Then, the map $\phi:M^2\rightarrow U(n)$ defined by
\[\phi=(\pi_{F_0}-\pi_{F_0}^\perp)(\pi_{1}-\pi_{1}^\perp)...(\pi_{r}-\pi_{r}^\perp)\]
is harmonic.

Further, all harmonic maps from $M^2$ to $G_*(\mathbb C^n)$ of finite uniton number, and so harmonic maps from $S^2$ to $G_*(\mathbb C^n)$, are obtained this way.
\end{theorem}

From now on we will represent the meromorphic data $K$, by $L$, when it takes values in $F_0$, or by $E$ if it take values in $F_0^{\perp}$.

\begin{example} For a general $n$, let $F_0$ be a two dimensional subspace, $r=3$ and $j=2$. Let $L_{i,1}\in F_0$, $E_{i,1}\in F_0^\perp$ ($0\leq i\leq 2$) and consider the $F_0$-array\renewcommand{\arraycolsep}{5pt}
\[\left[\begin{array}{ll}
L_{0,1}&E_{0,1}\\
E_{1,1}&L_{1,1}\\
L_{2,1}&E_{2,1}\\
\end{array}\right].\]\renewcommand{\arraycolsep}{1pt}\noindent

Then, using \eqref{Equation:FirstEquationInTheorem:MainTheoremNonBundleVersion}, one gets $H_{0,1}=L_{0,1}$, $H_{0,2}=E_{0,1}$, $H_{1,1}=E_{1,1}$, $H_{1,2}=L_{1,1}$, $H_{2,1}=-E_{1,1}+L_{2,1}$, and $H_{2,2}=-L_{1,1}+E_{2,1}$.

We will assume that $L_{0,1}$, $L_{0,1}^{(1)}$ are linearly independent and that $E_{0,1}$, $E_{0,1}^{(1)}$,
$E_{0,1}^{(2)}$ are also linearly independent.

Therefore, the map $\phi=(\pi_{F_0}-\pi_{F_0}^\perp)...(\pi_3-\pi_3^\perp)$ is harmonic and $G_*(\mathbb C^n)$-valued, where
\[\begin{array}{ll}
\underline{\alpha}_1=\wordspan \{&L_{0,1},E_{0,2}\}\\
\underline{\alpha}_2=\wordspan \{&L_{0,1}+\pi_1^\perp E_{1,1} ,E_{0,1}+\pi_1^\perp L_{1,1},\pi_1^\perp L^{(1)}_{0,1},\pi_1^\perp E^{(1)}_{0,1}\}\\
\underline{\alpha}_3=\wordspan \{&L_{0,1}+\pi_1^\perp E_{1,1}+\pi_2^\perp \pi_1 E_{1,1},E_{0,1}+\pi_1^\perp L_{1,1}+\pi_2^\perp\pi_1L_{1,1}+\pi_2^\perp\pi_1^\perp E_{2,1},\\
~&(\pi_1^\perp+\pi_2^\perp)L^{(1)}_{0,1}+\pi_2^\perp\pi_1^\perp E_{1,1}^{(1)},(\pi_1^\perp+\pi_2^\perp) E^{(1)}_{0,1}+\pi_2^\perp\pi_1^\perp E_{1,1}^{(1)},\pi_2^\perp\pi_1^\perp E_{0,1}^{(2)}\}
\end{array}\]
We remark that $\pi_{2}^{\perp}\pi_{1}^{\perp}L_{2,1}$ and  $\pi_{2}^{\perp}\pi_{1}^{\perp}L_{0,1}^{(2)}$ vanish, since $L_{0,1}$ and $\pi_{1}^\perp H_{0,1}^{(1)}$ span $F_0$. It is easily seen from the decomposition $\underline{F}_1 = \underline{\alpha}_{1}\cap F_0\oplus\underline{\alpha}_{1}^{\perp}\cap F_{0}^\perp$ that the rank of the bundle $F_1$ is $n-2$; in fact $\wordrank(\underline{\alpha}_1\cap F_0)= 1$ and $\wordrank(\underline{\alpha}_{1}^\perp\cap F_{0}^\perp)=n-3$.

In the same way we conclude that $\wordrank(\underline{F}_2)=2$, since $\underline{\alpha}_{2}^{\perp} \cap \underline{F}_{1}^\perp=\{0\}$. Then from the decomposition
$\underline{F}_3 = \underline{\alpha}_{3} \cap \underline{F}_{2} \oplus \underline{\alpha}_3^\perp \cap \underline{F}_2^\perp$ we obtain $\wordrank(\underline{F}_3)=n-3$ and $\phi$ is a harmonic map into $G_{n-3}(\mathbb C^n)$ with uniton number $3$.
\end{example}

From now on, for a harmonic map $\phi : M^2\rightarrow G_*(\mathbb C^n)$, we will represent by $\underline{F}_{\phi}$ the corresponding tautological bundle and for $\phi =(\pi_0-\pi_0^{\perp})(\pi_1-\pi_1^\perp)...(\pi_r-\pi_r^\perp)$ we will write $\underline{F}_{\phi_i}=\underline{F}_i$, where $\phi_i=(\pi_0-\pi_0^\perp)...(\pi_i-\pi_i^\perp)$ ($0\leq i\leq r$) and $F_0$ is a constante subspace of $\mathbb C^n$.

We let $h:G_k(\mathbb C^n) \rightarrow G_{n-k}(\mathbb C^n)$ represent the isometry given by $h(F)=F^{\perp}$. Of course, $\phi=\pi_{F_i}-\pi_{F_i}^\perp$ implies that $h(\phi)=\pi_{F_i}^\perp-\pi_{F_i}$. Hence,
\begin{proposition}\label{Proposition:Perping}
If $\phi_i=(\pi_0-\pi_0^{\perp})(\pi_1-\pi_1^{\perp})...(\pi_i-\pi_i^{\perp})$, then $h\circ\phi_i=(\pi_0^{\perp}-\pi_0)(\pi_1-\pi_1^{\perp})...(\pi_i-\pi_i^{\perp})$.
\end{proposition}


\subsection{Harmonic maps into $G_p(\mathbb C^n)$}\label{Subsection:HarmonicMapsIntoGpCn}


$\text{ }$

Given a subspace $F_0$ of $\mathbb C^n$ the main ingredient in building harmonic maps of finite uniton number is the selection of meromorphic data with values in $F_0$ and $F_0^\perp$. Let $k$ denote the dimension of the complex subspace $F_0$ of $\mathbb C^{n}$, $r$ the uniton number and fix $0\leq i\leq r-1$.

For each family $\{L_{i,j}\}_{1 \leq j \leq n}$ such that $L_{a,j}=0$ whenever $0\leq a<i$ we use the notation: $l_i^t= \wordrank \,\wordspan \{C^{i+t}_{i+t}L^{(t)}_{i,j}\}_{1 \leq j \leq n}$, where $0\leq t\leq r-i-1$.

Analogously, for each family $\{E_{i,j}\}_{1 \leq j \leq n}$ such that $E_{a,j}=0$ whenever $0\leq a <i$ we use the notation $s_i^t= \wordrank \,\wordspan \{C^{i+t}_{i+t}E^{(t)}_{i,j}\}_{1 \leq j \leq n}$, where $0 \leq t \leq r-i-1$.

In this way we get two triangular $r\times r$ matrices\renewcommand{\arraycolsep}{5pt}
\[L=\left[
\begin{array}{cccc}
l_{0}^{0} & 0 & \cdots  & 0 \\
l_{0}^{1} & l_{1}^{0} & \cdots & 0 \\
\vdots  & \vdots  & \ddots  & \vdots \\
l_{0}^{r-1} & l_{1}^{r-2} & \cdots  & l_{r-1}^{0}
\end{array}
\right] \text{ and }S=\left[
\begin{array}{cccc}
s_{0}^{0} & 0 & \cdots  & 0 \\
s_{0}^{1} & s_{1}^{0} & \cdots & 0 \\
\vdots  & \vdots  & \ddots  & \vdots \\
s_{0}^{r-1} & s_{1}^{r-2} & \cdots  & s_{r-1}^{0}
\end{array}
\right],\]\renewcommand{\arraycolsep}{1pt}\noindent
where, in each column $i+1$, the entries  $(l_{i}^{0},\cdots l_{i}^{r-i-1})$ and $(s_{i}^{0},\cdots s_{i}^{r-i-1})$ are decreasing sequences.

Notice that the sum of all entries of both matrices up to the $i$'th line is exactly the rank of $\alpha_{i+1}$. Of course, the sum of all entries of $L$ has to be less or equal than $k$, the sum of all entries of $S$ has to be less or equal than $n-k$ and the sum of all entries of both matrices has to be less or equal than $n-1$.

Under the above conditions we will say that \emph{the pair $(L,S)$ is adapted to $F_{0}$}. From now on, $F_0$ will denote a subspace of $\mathbb C^n$ with dimension $k$ and $(L,S)$ will represent an adapted pair of matrices of order $r$.

\begin{example}
Let $n=10$, $k=5$ and consider an $F_0$ array of the form\renewcommand{\arraycolsep}{5pt}
\[\left[
\begin{array}{ccccc}
L_{0,1} & E_{0,1} & 0 & 0 & 0 \\
E_{1,1} & L_{1,1} & E_{1,2} & L_{1,2} & 0 \\
L_{2,1} & E_{2,1} & L_{2,2} & E_{2,2} & L_{2,3}
\end{array} \right]\]\renewcommand{\arraycolsep}{1pt}\noindent
to build a uniton number $3$ harmonic map $\varphi: S^2 \rightarrow G_*(\mathbb C^{10})$, according to Theorem \ref{Theorem:MainTheoremNonBundleVersion}. We know that $\underline{\alpha} _1= \wordspan\left\{ L_{0,1},E_{0,1}\right\}$, and

\[\begin{array}{ll}
\underline{\alpha} _{2}^{(0)}=\wordspan\{&L_{0,1}+C_{1}^{1}E_{1,1},E_{0,1}+C_{1}^{1}L_{1,1},C_{1}^{1}L_{1,2},C_{1}^{1}E_{1,2}\},\\
\underline{\alpha} _{2}^{(1)}=\wordspan\{&C_{1}^{1}L_{0,1}^{(1)},C_{1}^{1}E_{0,1}^{(1)}\},\\
~&~\\
\underline{\alpha}_{3}^{(0)}=\wordspan\{&L_{0,1}+C_{1}^{2}E_{1,1}+C_{2}^{2}L_{2,1},E_{0,1}+C_{1}^{2}L_{1,1}+C_{2}^{2}E_{2,1},\\
~&C_{1}^{2}L_{1,2}+C_{2}^{2}E_{2,2},C_{1}^{1}E_{1,2}+C_{2}^{2}L_{2,2},C_{2}^{2}L_{2,3}\},\\
\underline{\alpha}_{3}^{(1)}=\wordspan\{&C_{1}^{2}L_{0,1}^{(1)}+C_{2}^{2}E_{1,1}^{(1)},C_{1}^{2}E_{0,1}^{(1)}+C_{2}^{2}L_{1,1}^{(1)},C_{2}^{2}E_{1,2}^{(1)}\},\\
\underline{\alpha} _{3}^{(2)}=\wordspan\{&C_{2}^{2}L_{0,1}^{(2)}\},
\end{array}\]
where we have assumed that $C_{2}^{2}L_{1,2}^{(1) }=C_{2}^{2}E_{0,1}^{(2) }=0$, $\wordrank(\underline{\alpha}_1)=2$, $\wordrank(\underline{\alpha}_2)=6$ and $\wordrank(\underline{\alpha}_3)=9$.

\bigskip
As we have seen before, underlying the construction of a uniton three harmonic map there is a pair  $(L,S)$ of $3 \times 3$ diagonal matrices adapted to $F_0$, say\renewcommand{\arraycolsep}{5pt}
\[L=\left[
\begin{array}{ccc}
l_{0}^{0} & 0 & 0 \\
l_{0}^{1} & l_{1}^{0} & 0 \\
l_{0}^{2} & l_{1}^{1} & l_{2}^{0}
\end{array}
\right]\text{ and }S=\left[
\begin{array}{ccc}
s_{0}^{0} & 0 & 0 \\
s_{0}^{1} & s_{1}^{0} & 0 \\
s_{0}^{2} & s_{1}^{1} & s_{2}^{0}
\end{array}
\right].\]\renewcommand{\arraycolsep}{1pt}\noindent
We remark that, for each $j \in \left\{0,1,2\right\}$ and $i\leq j$, $\wordrank(\underline{\alpha}_{j+1}^i)=\sum_{k=0}^{j-i}(l_{k}^{i}+s_{k}^{i})$, so that the rank of $\underline{\alpha}_{j+1}$ is $\sum_{i=0}^{j}\sum_{k=0}^{j-i}(l_{k}^{i}+s_{k}^{i})$.

Since $C_2^2 L_{1,2}^{(1)}=C_2^2 E_{0,1}^{(2)}=0$, in the particular case of this example we have\renewcommand{\arraycolsep}{5pt}
\[L=\left[
\begin{array}{ccc}
1 & 0 & 0 \\
1 & 1 & 0 \\
1 & 0 & 1
\end{array}
\right]\text{ and }S=\left[
\begin{array}{ccc}
1 & 0 & 0 \\
1 & 1 & 0 \\
0 & 1 & 0
\end{array}
\right].\]\renewcommand{\arraycolsep}{1pt}\noindent
\end{example}
\bigskip
In the sequel, given $F_0$ and an adapted pair $(L,S)$ of $r \times r$ matrices,  we will use the following notation: $A_0=B_0=0$ and, for each $i\in\left\{1,...,r-1\right\}$, $A_{i}=\sum_{r=0}^{i-1}l_{r}^{0}+s_{r}^{0}$ and $B_i=A_i+l_i^0$.

\begin{definition}\label{Definition:ArrayMatchingPair}
An $r\times n$ $F_0$-array $(K_{i,j})_{0\leq i\leq r-1, 1\leq j\leq n}$  is said to match the the ordered pair $(L,S)$ if, for each $i\in\{0,...,r-1\}$, the following conditions hold:

(i) $\pi_{F_0^{\perp}}(K_{i,j}) = 0,\,\forall\,A_i + 1 \leq j \leq A_i + l_i^0$ and $\pi_{F_0}(K_{i,j}) = 0,\,\forall\, B_i+1\leq j\leq B_i+s_i^0=A_{i+1}$.

(ii) For each $0 \leq j \leq i$, $\wordrank\,\wordspan \left\{C^i_i K^{(i-j)}_{j, A_{j}+1},...,C^i_i K^{(i-j)}_{j, A_{j}+l_j^0}\right\}=l_j^{i-j}$ and \\
$\wordrank\,\wordspan \left\{C^i_i K^{(i-j)}_{j, B_{j}+1},...,C^i_i K^{(i-j)}_{j, B_{j}+s_j^0}\right\}=s_j^{i-j}$.

(iii) $\wordrank\,\wordspan \left\{C^i_i K^{(i-j)}_{j, A_{j}+1},...,C^i_i K^{(i-j)}_{j, A_{j}+l_j^0}\right\}_{0\leq j \leq i}=\displaystyle{\sum _{j=0}^{i}l_j^{i-j}}$ and

$\wordrank\,\wordspan\left\{C^i_i K^{(i-j)}_{j, B_{j}+1},...,C^i_i K^{(i-j)}_{j, B_{j}+s_j^0}\right\}_{0\leq j \leq i}=\displaystyle{\sum _{j=0}^{i}s_j^{i-j}}$.
\end{definition}

\begin{remark}
(i) We easily conclude that the rank of the bundle $\underline{\alpha}_{i+1}$ is $\sum_{t=0}^{i}\sum_{j=0}^{i-t}(l_{t}^{j}+s_{t}^{j})$, the sum of all entries of the first $i$ lines of both triangular matrices.

(ii) The $l_i^t$ and $s_i^t$ are independent of the choice of the complex coordinate $z$; in fact, once $\alpha_{j+1}^{(0)}$ is defined, letting
\[V_j=\wordspan\{C_j^j K_{j,A_j+1}^{(i-j)},...,C_j^jK_{j,A_j+l_j^0}^{(i-j)}\}=\left\{
\begin{array}{l}
\ker \pi_j|_{\alpha_{j+1}^{(0)}}\cap F_{j+1}^\perp,\text{ if $j$ odd}\\
\ker \pi_j|_{\alpha_{j+1}^{(0)}}\cap F_{j+1},\text{ if $j$ even},
\end{array}
\right.\]
we have $l_j^0=\wordrank V_j$ and $l_j^{i-j}=\wordrank A_z^{\phi_{i}}...A_z^{\phi_{j+1}}V_j$ ($i\geq j$). Analogously with respect to the $s_i^t$.
\end{remark}

An induction argument allows the following result:

\begin{theorem}\label{Theorem:MainTheoremForUnitonCounting}
Let $r \in \{1,...,n-1)$, $F_0$ be a $k$-dimensional subspace of $\mathbb C^n$ and consider a pair $(L,S)$ adapted to $F_0$. For any $F_0$-array $(K_{i,j})_{0\leq i\leq r-1, 1\leq j\leq n}$ which matches $(L,S)$ and $i\in\{0,...,r\}$, the rank of the tautological bundle $\underline{F}_i$ corresponding to the harmonic map $\phi_i=(\pi_0 -\pi_{0}^{\perp})...(\pi_i -\pi_{i}^{\perp})$ is given by
\begin{equation}\label{Equation:EquationInTheorem:MainTheoremForUnitonCounting}
\left\{\begin{array}{l}\displaystyle{k+\sum_{j=0}^{\frac{i}{2}-1}\sum_{t =0}^{2j+1}
(s_{t}^{2j+1-t}-l_{t}^{2j+1-t})}\text{ if }i\text{ is even}\\
\displaystyle{n-\Big[k+ \sum_{j=0}^{\frac{i-1}{2}}\sum_{t =0}^{2j}
(s_{t}^{2j-t}-l_{t}^{2j-t})\Big] }\text{ if }i\text{ is odd,}
\end{array}\right.
\end{equation}

\end{theorem}

Using Theorem \ref{Theorem:MainTheoremForUnitonCounting} we can see that, when we start with a harmonic map $\phi:M^2 \rightarrow G_p(\mathbb C^n)$ and add a uniton $\underline{\alpha}$, the harmonic map $\phi(\pi_{\underline{\alpha}}-\pi_{\underline{\alpha}^{\perp}})$ does not, in general, take values in the same Grassmannian. However, in certain cases, it is possible to add unitons in such a way that the successive harmonic maps stay in the same Grassmannian (see Example \ref{Example:FirstExampleAfter:Theorem:MainTheoremForUnitonCounting}).

\begin{example}\label{Example:FirstExampleAfter:Theorem:MainTheoremForUnitonCounting}
Let us consider $G_{^4}(\mathbb C^8)$ as target manifold and start with a $4$-dimensional complex subspace $F_0$ of $\mathbb C^8$. We select the ordered pair $(L,S)$ adapted to $F_0$ with\renewcommand{\arraycolsep}{5pt}
\[L=S=\left[
\begin{array}{ccc}
1 & 0 & 0 \\
1 & 0 & 0 \\
1 & 0 & 0
\end{array}
\right]\]\renewcommand{\arraycolsep}{1pt}\noindent
and take a $F_0$-array which matches the pair $(L,S)$. Then the harmonic maps $\phi_1$, $\phi_2$ and $\phi_3$ all have values in $G_4(\mathbb C^8)$, as it is easily seen from Theorem \ref{Theorem:MainTheoremForUnitonCounting}, since $l_0^j=s_0^j$ for every
$j\in\left\{0,1,2\right\}$.
\end{example}

\begin{example}\label{Example:SecondExampleAfter:Theorem:MainTheoremForUnitonCounting}
In this example, using Theorem \ref{Theorem:MainTheoremForUnitonCounting}, we will describe all harmonic maps $\phi:S^2 \rightarrow G_2(\mathbb C^5)$ with uniton number $3$. Let $F_0$ be a $k$- dimensional complex subspace of $\mathbb C^5$ ($0\leq k\leq 5$) and $(L,S)$ a pair adapted to $F_0$,\renewcommand{\arraycolsep}{5pt}
\[L=\left[
\begin{array}{ccc}
l_0^0 & 0 & 0 \\
l_0^1 & l_1^0 & 0 \\
l_0^2 & l_1^ 1 & l_2^0
\end{array}
\right],\,S=\left[
\begin{array}{ccc}
s_0^0 & 0 & 0 \\
s_0^1 & s_1^0 & 0 \\
s_0^2 & s_1^ 1 & s_2^0
\end{array}
\right].\]\renewcommand{\arraycolsep}{1pt}\noindent
The sum of all entries of both matrices has to be less or equal than $4$ and the uniton number three condition implies that at least one element of the third lines of the matrices has to be different from zero. From Theorem \ref{Theorem:MainTheoremForUnitonCounting} we know that
\begin{equation}\label{Equation:EquationInExample:SecondExampleAfter:Theorem:MainTheoremForUnitonCounting}
2=5 -[k+(s_0^0-l_0^0)+(s_0^2-l_0^2)+(s_1^1-l_1^1)+(s_2^0-l_2^0)].
\end{equation}
We have to analyze the different cases according to the dimension of $F_0$.
(a) Considering $k=5$, i.e, $S=0$, we have $2=l_0^0+l_0^2+l_1^1+l_2^0$. The only possibility is\renewcommand{\arraycolsep}{5pt}
\[L=\left[
\begin{array}{ccc}
1 & 0 & 0 \\
1 & 0 & 0 \\
1 & 0 & 0
\end{array}
\right],\]\renewcommand{\arraycolsep}{1pt}\noindent
since $\underline{\alpha}_1$ is full. This gives rise to the unitons
\[\begin{array}{ll}\underline{\alpha}_1&=\wordspan \{L_{0,1}\},\\
\underline{\alpha}_2&=\wordspan\{L_{0,1},\pi_1^\perp L_{0,1}^{(1)}\}\text{ and}\\
\underline{\alpha}_3&=\wordspan\{L_{0,1},\pi_1^\perp L_{0,1}^{(1)},\pi_2^\perp\pi_1^\perp L_{0,1}^{(2)}\}.\end{array}\]

(b) Now we analise the case $k=4$. Here we have $1=(l_0^0-s_0^0)+(l_0^2-s_0^2)+(l_1^1-s_1^1)+(l_2^0-s_2^0)$. It is not hard to check that the only possibility is\renewcommand{\arraycolsep}{5pt}
\[L=\left[
\begin{array}{ccc}
1 & 0 & 0 \\
1 & 0 & 0 \\
1 & 0 & 0
\end{array}
\right]\text{ and }S=\left[
\begin{array}{ccc}
1 & 0 & 0 \\
0 & 0 & 0 \\ 0 & 0 & 0
\end{array}
\right].\]\renewcommand{\arraycolsep}{1pt}\noindent
Therefore we choose our meromorphic data $L_{0,1}$, $L_{1,1}$ and $L_{2,1}$ with values in $F_0$ and $E_{0,1}$ with values in $F_{0}^\perp$. This corresponds to
\[\begin{array}{ll}
\underline{\alpha}_1&=\wordspan\{L_{0,1}, E_{0,1}\}, \\
\underline{\alpha}_2&=\wordspan\{L_{0,1}, E_{0,1}+\pi_1^\perp L_{1,1}, \pi_1^\perp L_{0,1}^{(1)}\}\text{ and}\\
\underline{\alpha}_3&=\wordspan\{L_{0,1}+\pi_2^\perp\pi_1^\perp L_{2,1}, E_{0,1}+C_1^2 L_{1,1},\pi_1^\perp L_{0,1}^{(1)},\pi_2^{\perp}\pi_1^\perp L_{0,1}^{(2)}\}.
\end{array}\]

(c) Consider $k=0$, which corresponds to $L=0$ and implies $3=s_0^0+s_0^2+s_1^1+s_2^0$. We remark that cases like\renewcommand{\arraycolsep}{5pt}
\[S=\left[
\begin{array}{ccc}
1 & 0 & 0 \\
1 & 0 & 0 \\
0 & 0 & 2
\end{array}
\right]\text{ or }S=\left[
\begin{array}{ccc}
2 & 0 & 0 \\
0 & 0 & 0 \\
0 & 0 & 1
\end{array}
\right],\]\renewcommand{\arraycolsep}{1pt}\noindent
although satisfy our equation, have to be excluded, since do not fulfil the fullness of $\underline{\alpha}_1$. Hence the only possibility is\renewcommand{\arraycolsep}{5pt}
\[S=\left[
\begin{array}{ccc}
1 & 0 & 0 \\
1 & 0 & 0 \\
1 & 0 & 1
\end{array}
\right]\]\renewcommand{\arraycolsep}{1pt}\noindent
and we choose our meromorphic data $E_{0,1}$, $E_{2,1}$ and $E_{2,2}$ with values in $F_{0}^{\perp}$ to get the harmonic map $\phi=(\pi_0-\pi_0^{\perp})(\pi_1-\pi_1^{\perp})(\pi_2-\pi_2^{\perp})(\pi_3-\pi_3^{\perp})$, where
\[\begin{array}{ll}
\underline{\alpha}_1&=\wordspan\{E_{0,1}\},\\
\underline{\alpha}_2&=\wordspan\{ E_{0,1},\pi_1^\perp E_{0,1}^{(1)}\}\text{ and}\\
\underline{\alpha}_3&=\wordspan\{E_{0,1}+\pi_2^{\perp}\pi_1^{\perp}E_{2,1},\pi_1^{\perp}E_{0,1}^{(1)}, \pi_2^{\perp}\pi_1^{\perp}E_{0,1}^{(2)}, \pi_2^{\perp}\pi_1^{\perp}E_{2,2} \}.\end{array}\]

The cases $k=1,2,3$ must be excluded. Regarding $k=2,3$, the fullness of $\underline{\alpha}_1$ would imply that\renewcommand{\arraycolsep}{5pt}
\[L_1=\left[
\begin{array}{c}
1\\
1\\
0
\end{array}
\right]\text{ and }S_1=\left[
\begin{array}{c}
1\\
1\\
1
\end{array}
\right]\]\renewcommand{\arraycolsep}{1pt}\noindent
in the first case and\renewcommand{\arraycolsep}{5pt}\noindent
\[L_1=\left[
\begin{array}{c}
1\\
1\\
1
\end{array}
\right]\text{ and }S_1=\left[
\begin{array}{c}
1\\
1\\
0
\end{array}
\right]\]\renewcommand{\arraycolsep}{1pt}\noindent
in the second case, which is not adequate for $G_2(\mathbb C^5)$ as the sum of these entries is $5$. As for the case $k=1$, the fullness of $\underline{\alpha}_1$ would require\renewcommand{\arraycolsep}{5pt}
\[L=\left[
\begin{array}{ccc}
1 & 0 & 0 \\
0 & 0 & 0 \\
0 & 0 & 0
\end{array}
\right]\text{ and }S=\left[
\begin{array}{ccc}
1 & 0 & 0 \\
1 & 0 & 0 \\
1 & 0 & 0
\end{array}
\right],\]\renewcommand{\arraycolsep}{1pt}\noindent
which does not satisfy \eqref{Equation:EquationInExample:SecondExampleAfter:Theorem:MainTheoremForUnitonCounting}. Hence, the three cases (a), (b) and (c) are the only ones yielding uniton number three harmonic maps into $G_2(\mathbb C^5)$.

From Proposition \ref{Proposition:Perping} interchanging $F_0^{\perp}$ with $F_0$ and $S$ with $L$, we get the description of all uniton number three harmonic maps into $G_3(\mathbb C^5)$.

\end{example}

It is known that, for a harmonic map $\phi: M^2 \rightarrow G_p(\mathbb C^n)$, the maximal uniton number is less or equal than $2\text{min}\left\{p,n-p\right\}$ \cite{DongShen:96,BurstallGuest:97}. We will see, later on, that this estimate is sharp only when $n\neq 2p$ and $F_0$ is trivial.

In the next theorem, fixing a subspace $F_0$ with dimension $k$, we present an estimate for the uniton number of a harmonic map
$\phi=(\pi_0-\pi_0^{\perp})(\pi_1-\pi_1^{\perp})...(\pi_i-\pi_i^{\perp})$, when $2p\leq n$. This estimate is sharp and covers all situations, for if $n>p$ and $\phi : M^2 \rightarrow G_p(\mathbb C^n)$ is harmonic,
$h\circ \phi: M^2 \rightarrow G_{n-p}(C^n)$ is a harmonic map with the same uniton number and $2(n-p)<n$.

\begin{theorem}\label{Theorem:EstimatesOnTheUnitonNumber}
Let $F_0$ be a $k$-dimensional complex subspace of $\mathbb C^n$ and $\phi=(\pi_0-\pi_0^{\perp})(\pi_1-\pi_1^{\perp})...(\pi_{r_k}-\pi_{r_k}^\perp)$ be a harmonic map into $G_p(\mathbb C^n)$,  where $r_k$ is the uniton number and $2p \leq n$. Then,

(i) $r_k\leq \min \left\{ 2p-k-a_{k},n-1\right\}$if $k<p$;

(ii) $r_k\leq p$ if $k\geq p$ and $k+p\leq n$;

(iii) $r_k\leq 2p-(n-k)-a_{k}$ if $k\geq p$ and $k+p>n$,

where $a_{k}=\left\{\begin{array}{l}
1\text{ if $k$ is even and $k<p$ or $n-k$ is even and $k\geq p$}\\
0\text{ if $k$ is odd and $k<p$ or $n-k$ is odd and $k\geq p$}.
\end{array}
\right.$

Moreover, the above estimates are sharp, except in the case $k=p$, where $r_k \leq p-1$.
\end{theorem}

A glance at the list of possibilities given by the previous proposition allows to verify that the maximal uniton number is realized when $k=0$ and $2p\leq n$, or when $k=n$ and $2p\geq n$.

\begin{example}
Assume $\phi=(\pi_0-\pi_0^{\perp})...(\pi_r-\pi_r^{\perp}) : S^2 \rightarrow G_2(\mathbb C^4)$, where $k=2$. From Theorem \ref{Theorem:EstimatesOnTheUnitonNumber} we know that, for $k=2$ fixed, the maximal uniton number is $2$. Let us describe those harmonic maps. Consider an adapted pair $(L,S)$ of ordered $2\times 2$ diagonal matrices adapted to $F_0$. From \eqref{Equation:EquationInTheorem:MainTheoremForUnitonCounting}, we get

$2=2+\sum_{j=0}^{1}(s_j^{1-j}-l_j^{1-j})$ or $s_o^1 + s_1^0 = l_0^1 + l_1^1$.

Clearly the only possibility is\renewcommand{\arraycolsep}{5pt}
\[L=S=\left[
\begin{array}{cc}
1 & 0 \\
1 & 0
\end{array}
\right].\]\renewcommand{\arraycolsep}{1pt}\noindent
Thus we see that we shall start with meromorphic data $L_0$ and $E_0$ with values in $F_0$ and $F_0^\perp$, respectively, to build the unitons
\[\begin{array}{ll}\underline{\alpha}_1&=\wordspan\{L_0, E_0\}\text{ and}\\
\underline{\alpha}_2&= \wordspan\{L_0, E_0, \pi_1^{\perp}L_0^{(1)}, \pi_1^\perp E_0^{(1)}\}.\end{array}\]
\end{example}

\begin{example}
Let us now describe the construction of all harmonic maps $\phi:S^2 \rightarrow G_3(\mathbb C^8)$ (respectively
$\phi:S^2 \rightarrow G_5(\mathbb C^8)$) of the type $\phi=(\pi_0-\pi_0^{\perp})(\pi_1-\pi_1^{\perp})...(\pi_r-\pi_r^{\perp})$, where $k=4$ and $r_k$ is maximal.

We know from Theorem \ref{Theorem:EstimatesOnTheUnitonNumber} that $r_k=3$. Hence, using Theorem \ref{Theorem:MainTheoremForUnitonCounting}, we get $1=(l_0^0 + l_0^2 + l_1^1 + l_2^0)-(s_0^0 + s_0^2 + s_1^1 + s_2^0)$.

Let us try to describe the possible pairs $(L,S)$ of diagonal $3 \times 3$ matrices adapted to $F_0$. As above, since $\alpha_1$ is full, we must have $l_0^i \neq 0$ and $s_0^i \neq 0$ for every $i\in \left\{1,2,3\right\}$. It is easily seen that the only possibility is\renewcommand{\arraycolsep}{5pt}
\[L=\left[
\begin{array}{ccc}
1 & 0 & 0 \\
1 & 0 & 0 \\
1 & 0 & 1
\end{array}
\right]\text{ and }S=\left[
\begin{array}{ccc}
1 & 0 & 0 \\
1 & 0 & 0 \\
1 & 0 & 0
\end{array}
\right].\]\renewcommand{\arraycolsep}{1pt}\noindent
Therefore we must choose our meromorphic data, $\left\{L_{0,1}, L_{1,1}, L_{2,1}, L_{2,2}\right\}$ with values in $F_0$, and
$\left\{E_{0,1}, E_{1,1}, E_{2,1}\right\}$ with values in $F_0^\perp$. This gives rise to the unitons

\[\begin{array}{ll}
\alpha _1&=\wordspan\{L_{0,1}, E_{0,1}\},\\
\alpha _2&=\wordspan\{L_{0,1}+C_1^1 E_1^1, E_{0,1}+C_1^1L_{1,1}, C_1^1L_{0,1}^{(1)}, C_1^1E_{0,1}^{(1)}\},\\
\alpha _3&=\wordspan\{L_{0,1}+C_2^1 E_1^1 +C_2^2L_{2,1}, E_{0,1}+C_1^1L_{1,1}+C_2^2E_{2,1}\},\\
~&\hspace{15mm}C_2^1L_{0,1}^{(1)}+C_2^2E_{1,1}^{(1)},C_1^1E_{0,1}^{(1)}+C_2^2L_{1,1}^{(1)},C_2^2L_{2,2},C_2^2L_{0,1}^{(2)},C_2^2E_{0,1}^{(2)}\}.\end{array}\]

Interchanging $L$ and $S$ and choosing the same holomorphic data we get the description of all harmonic maps $\phi:S^2\rightarrow C_5(\mathbb C^8)$ with uniton number $3$.
\end{example}

We may synthesize the above results in the following statement concerning harmonic maps $\phi:S^2\to G_p(\mathbb C^n)$ into a fixed Grassmannian.

\begin{theorem}\label{Theorem:SynthesizedVersion}
Let $q=\min\{p,n-p\}$, $k\in\{0,...,n\}$ and $r_k$ be under the conditions of Theorem \ref{Theorem:EstimatesOnTheUnitonNumber}. Taking a pair $(L,S)$ of $i\times i$ matrices ($1\leq i\leq r_k$) adapted to $F_0$, whose entries satisfy equations \eqref{Equation:EquationInTheorem:MainTheoremForUnitonCounting}, and an array $(K_{i,j})$ matching $(L,S)$, the map
\[\phi_i=\left\{\begin{array}{l}
(\pi_0-\pi_0^\perp)(\pi_1-\pi_1^\perp)...(\pi_i-\pi_i^\perp)\text{, if $q=p$}\\
(\pi_0^\perp-\pi_0)(\pi_1-\pi_1^\perp)...(\pi_i-\pi_i^\perp)\text{, if $q=n-p$}\\
\end{array}\right.\]
is a harmonic map into $G_p(\mathbb C^n)$. Moreover, all harmonic maps $\phi:S^2\to G_p(\mathbb C^n)$ are obtained this way.
\end{theorem}
%


\subsection{A note on the Grassmannian model}\label{Subsection:GrassmannianModel}


$\text{ }$

Let $\H$ denote the Hilbert space $L^2(S^1,\fieldC^n)$ and let $\H_+$ denote the linear closure of elements of the form $\sum_{k \geq 0}\lambda^k e_j$ where $\{e_j\}_{1\leq j\leq n}$ form the standard basis of $\fieldC^n$. The algebraic loop group $\Omega^{\wordalg}\U(n)$ consists of maps $\gamma:S^1\to \U(n)$ with $\gamma(1)=I$ and such that $\gamma(\lambda)=\sum_{k=0}^r\lambda^k A_j$, for some integer $r$ and $A_j\in\gl(n,\fieldC)$. It acts naturally on $\H$ and the correspondence $\gamma\to\gamma(H_+)$ identifies $\Omega^{\wordalg}\U(n)$ with the \emph{algebraic Grassmannian} consisting of all subspaces $W$ of $\H$ such that $\lambda W\subseteq W$ and $\lambda^r\H_+\subseteq W\subseteq\H_+$ for some $r$ \cite{PressleySegal:86,Segal:89}. In particular, we may identify $W$ with the coset $W+\lambda^r\H_+$ in the finite-dimensional vector space $\H_+/\lambda^r\H_+$; this vector space is canonically identified with $\fieldC^{rn}$ via the isomorphism
\begin{equation}\label{Equation:IsomorphismWithTheLs}
(R_0,R_1,...,R_{r-1})\to R_0+\lambda R_1+...+\lambda^{r-1}R_{r-1}+\lambda^r\H_+.
\end{equation}

Now, let $\phi:M^2\to\U(n)$ be a harmonic map of uniton number at most $r$ and $\Phi$ be its unique type one (polynomial) extended solution. We may naturally interpret $\Phi$ as a smooth map $\Phi:M^2\to\Omega^{\wordalg}\U(n)$. With the above identifications, we then have a holomorphic map $W=\Phi(\H_+)$ from $M^2$ into the into $G_*(\fieldC^{rn})$. Equivalently \cite{Guest:97}, a holomorphic subbundle $\underline{W}$ of the trivial bundle $M^2\times\fieldC^{rn}$ satisfying
\begin{equation}\label{Equation:WHarmonicityCondition}
\lambda \underline{W}_{(1)} \subseteq \underline{W},
\end{equation}
where $\underline{W}_{(i)}$ ($i\geq 0$) denotes the subbundle spanned by (local) sections of $\underline{W}$ and their first $i$ derivatives with respect to any complex coordinate $z$ on $M^2$. We call $\underline{W}$ the \emph{Grassmannian model} of $\phi$ (or $\Phi$).

All such subbundles $\underline{W}$ are given by taking an arbitrary holomorphic subbundle $\underline{X}$ of $\underline{\fieldC}^{rn}$ and setting $\underline{W}$ equal to the coset \cite{Guest:02}
\begin{equation}\label{Equation:DefinitionOfWFromX}
\underline{W}=\underline{X}+\lambda\underline{X}_{(1)}+\lambda^2\underline{X}_{(2)}+\cdots+\lambda^{r-1}\underline{X}_{(r-1)}.
\end{equation}

For any $i\geq 0$ and meromorphic vectors $(H_0,H_1,\ldots, H_i)$, set
\begin{equation} \label{Equation:HToL}
R_i=\sum_{l=0}^i\binom{i}{l}H_{l}\,.
\end{equation}

The isomorphism \eqref{Equation:IsomorphismWithTheLs} allows us to describe the Grassmannian model of a finite uniton number harmonic map $\phi:M^2\to\U(n)$ in the following way \cite{FerreiraSimoesWood:09}:

\begin{theorem}
Let $r\geq 1$, and let $\underline{B}$ and $\underline{X}$ be holomorphic subbundles of $\fieldC^{rn}$ related by the linear isomorphism
\[\underline{B}\ni H = (H_0,H_1\ldots, H_{r-1}) \to R = (R_0,R_1,\ldots,R_{r-1}) \in \underline{X}\]
given by \eqref{Equation:HToL}. Write
\[\underline{\alpha}_{i+1}^{(k)}=\{\sum_{s=k}^iC^i_s H_{s-k}^{(k)},\,H\in\Gamma_{\rm{hol}}(\underline{B})\}\text{ and }\underline{\alpha}_{i+1}=\sum_{k=0}^i\underline{\alpha}_{i+1}^{(k)}.\]
Let $\phi:M^2 \to \U(n)$ be the harmonic map given by \eqref{Equation:SecondEquationInTheorem:Theorem11FromFerSimWood09} and $W:M^2 \to G_*(\fieldC^{rn})$ be the holomorphic map given by \eqref{Equation:DefinitionOfWFromX}. Then $W$ is the Grassmannian model of $\phi$.
\end{theorem}

Let $F_0$ denote a constant subspace in $\fieldC^n$. We say that a polynomial $R\in \H_+/\lambda^r \H_+$ is \emph{$F_0$-adapted} if its coefficients have image alternately in $F_0$ and $F_0^\perp$, i.e., $L(\lambda)=\sum_{i=0}^{r-1}L_i\lambda^i$ and \emph{either}

{\rm (i)} $R_i$ has image in $F_0$ for $i$ even, and in $F_0^\perp$ for $i$ odd, \emph{or}

{\rm (ii)} $R_i$ has image in $F_0^\perp$ for $i$ even, and in $F_0$ for $i$ odd.

Note that when $F_0$ is trivial, a polynomial is $F_0$-adapted if and only if it is even or odd, i.e, has coefficients of all odd or all even powers of $\lambda$ equal to zero. Using the Grassmannian model, we have the following characterization of maps into $G_*(\mathbb C^n)$ \cite{FerreiraSimoesWood:09}:

\begin{proposition}
$\Phi$ is the extended solution of a harmonic map into a Grassmannian if and only if $W$ has a spanning set consisting of $F_0$-adapted polynomials, or, equivalently, $W$ is given by \eqref{Equation:DefinitionOfWFromX} for some $X$ which has a spanning set consisting of $F_0$-adapted polynomials.
\end{proposition}

Now, let $\phi$ be a Grassmannian-valued harmonic map given as in Theorem \ref{Theorem:MainTheoremNonBundleVersion} for some $F_0$-array $(K_{i,j})_{0\leq i\leq r-1, 1\leq j\leq n}$. Using \eqref{Equation:FirstEquationInTheorem:MainTheoremNonBundleVersion} and \eqref{Equation:HToL} one can easily check that
\[R_{i,j}=\sum_{s=0}^i K_{s,j}.\]
Hence, the Grassmannian model for $\phi$ is given by
\[\underline{W}=\underline{X}+\lambda\underline{X}_{(1)}+\lambda^2\underline{X}_{(2)}+\cdots+\lambda^{r-1}\underline{X}_{(r-1)},\]
where $\underline{X}=(K_{0,j},K_{0,j}+K_{1,j},...,K_{0,j}+...+K_{r-1,j})$. Since
\[\begin{array}{ll}K_{0,j}+\lambda K_{1,j}+...+\lambda^{r-1}K_{r-1,j}&=K_{0,j}+\lambda(K_{0,j}+K_{1,j})+...+\lambda^{r-1}(K_{0,j}+...+K_{r-1,j})\\
~&-\lambda(K_{0,j}+\lambda(K_{0,j}+K_{1,j})+...+\lambda^{r-1}(K_{0,j}+...+K_{r-1,j}),
\end{array}\]
we can also write $\underline{W}=\underline{\tilde{X}}+\lambda\underline{\tilde{X}}_{(1)}+\lambda^2\underline{\tilde{X}}_{(2)}+\cdots+\lambda^{r-1}\underline{\tilde{X}}_{(r-1)}$, where $\underline{\tilde{X}}=(K_{0,j},...,K_{r-1,j})$: it is clear that $\underline{\tilde{X}}$ is spanned by $F_0$-adapted polynomials.

\medskip

Hence, we now can easily construct explicitly our harmonic map $\phi$ from its Grassmannian model $W$: given a set $R_{i,j}$ of $F_0$-adapted polynomials that generate $\underline{W}$, we set $K_{i,j}=R_{i,j}$ and construct the map $\phi$ as in Theorem \ref{Theorem:MainTheoremNonBundleVersion}.


\subsection{Proof of the main results}\label{Subsection:Proofs}


$\text{ }$

Let $\phi:M^2\to G_*(\fieldC^n)$ be a harmonic map and $\underline{\alpha}$ a uniton for $\phi$. From \cite{Uhlenbeck:89}, we know that $\tilde{\phi}=\phi(\pi_\alpha-\pi_\alpha^\perp)$ lies in $G_*(\fieldC^n)$ if and only if $\pi_\alpha$ and $\phi$ commute. This means that $\underline{F}_\phi$ splits the eigenspaces of $\pi_\alpha$ so that $\underline{\alpha}=\underline{\alpha}\cap\underline{F}_\phi\oplus\underline{\alpha}\cap\underline{F}_\phi^\perp$. As a consequence,
\[\tilde{\phi}=\pi_{\alpha\cap F_\phi\oplus\alpha^\perp\cap F_\phi^\perp}-\pi_{\alpha^\perp\cap F_\phi\oplus\alpha\cap F_\phi^\perp}.\]

Recall from \cite{Uhlenbeck:89} the following facts.

\begin{proposition}\label{Proposition:ComingBackStillInGStar}
Let $\phi:M^2\to G_*(\fieldC^n)$ be a harmonic map of finite uniton number $r$. Then, there are unique proper unitons $\underline{\alpha}_1,...,\underline{\alpha}_r$ satisfying the covering condition \eqref{Equation:CoverCondition} and with $\underline{\alpha}_1$ full and a constant map $Q=(\pi_{F_0}-\pi_{F_0}^\perp)$ such that
\begin{equation}\label{Equation:EquationInProposition:ComingBackStillInGStar}
\phi=(\pi_{F_0}-\pi_{F_0}^\perp)(\pi_1-\pi_1^\perp)...(\pi_r-\pi_r^\perp).
\end{equation}
Moreover, each partial map $\phi_{r'}=(\pi_{F_0}-\pi_{F_0}^\perp)(\pi_1-\pi_1^\perp)...(\pi_{r'}-\pi_{r'}^\perp)$ maps into $G_*(\fieldC^n)$
and commutes with $\pi_{r'+1}$.
\end{proposition}

In the sequel, we shall always consider a harmonic map $\phi:M^2\to G_*(\fieldC^n)$ factorized as in Proposition \ref{Proposition:ComingBackStillInGStar}. Let $\phi$ be as in \eqref{Equation:EquationInProposition:ComingBackStillInGStar}. We define recursively
\begin{equation}\label{Equation:SecondEquationInProposition:SequenceOfMapsAndSubspacesForTheGeneralCase}
\begin{array}{rll}
\underline{F}_1&=&F_0\cap\underline{\alpha}_1\oplus F_0^\perp\cap\underline{\alpha}_1^\perp;\\
\underline{F}_2&=&\underline{F}_1\cap\underline{\alpha}_2\oplus \underline{F}_1^\perp\cap\underline{\alpha}_2^\perp;\\
...\\
\underline{F}_{r}&=&\underline{F}_{r-1}\cap\underline{\alpha}_{r}\oplus\underline{F}_{r-1}^\perp\cap\underline{\alpha}_{r}^\perp.
\end{array}
\end{equation}
It is easy to check that $\phi$ lies in $G_*(\fieldC^n)$ if and only if $\underline{F}_i$ decomposes $\underline{\alpha}_{i+1}$ for all $1\leq i\leq r-1$. Moreover, in that case, each of the partial maps
\[\phi_{r'}=(\pi_{F_0}-\pi_{F_0}^\perp)(\pi_1-\pi_1^\perp)...(\pi_{r'}-\pi_{r'}^\perp)\]
also lies in $G_*(\fieldC^n)$ and $\phi_{r'}=\pi_{\underline{F}_{r'}}-\pi_{\underline{F}_{r'}}^\perp$.

\begin{corollary}\label{Corollary:CorollaryToLemma:TheGeneralCase}
$A_z^\phi$ interchanges $\underline{F}_r$ and $\underline{F}_r^\perp$.
\end{corollary}
\begin{proof}Recall that $2A_z^\phi=\phi^{-1}\partial_z\phi=(\pi_{\underline{F}_r}-\pi_{\underline{F}_r}^\perp)^{-1}\partial_z(\pi_{\underline{F}_r}-\pi_{\underline{F}_r}^\perp)$. In particular, if $f$ is a section of $\underline{F}_r$, we have that
\[\begin{array}{ll}
2A_z^\phi f&=(\pi_{\underline{F}_r}-\pi_{\underline{F}_r}^\perp)\partial_z(\pi_{\underline{F}_r}-\pi_{\underline{F}_r}^\perp)f=(\pi_{\underline{F}_r}-\pi_{\underline{F}_r}^\perp)(\partial_zf-\pi_{\underline{F}_r}\partial_zf+\pi_{\underline{F}_r}^\perp\partial_zf)\\
~&=2(\pi_{\underline{F}_r}-\pi_{\underline{F}_r}^\perp)(\pi_{\underline{F}_r}^\perp\partial_zf)\in \underline{F}_{r}^\perp.
\end{array}\]
If $f$ is a section of $\underline{F}_r^\perp$, the argument is similar.
\end{proof}

Let $S^i_j$ denote the sum of all ordered $i$-fold products of the form $\Pi_i\cdots\Pi_1$, where exactly $j$ of the $\Pi_{l}$ are $\pi_{l}^{\perp}$ and the other $i-j$ are $\pi_{l}$.  For $i=0$, set $S^{i}_{j}=I$, for $i<0$ or $j>i>0$, set $S^{i}_{j}=0$. Then, \cite{FerreiraSimoesWood:09}
\[S^i_j=\pi_i S^{i-1}_j+\pi_i^\perp S^{i-1}_{j-1}\]
and the $S^i_j$ are related with the $C^i_j$ by the formulae
\[C^i_k=\sum_{s=k}^i\binom{s}{k}S^i_s,\]
where $\binom{i}{s}$ denotes the binomial coefficient $i!/s!(i-s)!$.

\begin{lemma}\label{Lemma:ProjectionsOfSijGeneralCase}
Let $\underline{F}_i$, $1\leq i\leq r$ be defined as in \eqref{Equation:SecondEquationInProposition:SequenceOfMapsAndSubspacesForTheGeneralCase}. Then, if $\underline{F}_i$ decomposes $\underline{\alpha}_{i+1}$ for all $1\leq i\leq r-1$,
\begin{equation}\label{Equation:FirstEquationIn:Lemma:ProjectionsOfSij}
\begin{array}{ll}
S^i_j\pi_{F_0} A\in \underline{F}_i&\text{ if $j$ even}\\
S^i_j\pi_{F_0} A\in \underline{F}_i^\perp&\text{ if $j$ odd}\\
S^i_j\pi_{F_0}^\perp A\in \underline{F}_i^\perp&\text{ if $j$ even}\\
S^i_j\pi_{F_0}^\perp A\in \underline{F}_i&\text{ if $j$ odd},\\
\end{array}
\end{equation}
for any $A$.
\end{lemma}


\begin{proof} For the case $r=1$, assume that $F_0$ splits $\underline{\alpha}_1$. Let us show that \eqref{Equation:FirstEquationIn:Lemma:ProjectionsOfSij} holds. As a matter of fact:
\[\begin{array}{ll}
S^1_0\pi_{F_0} A=\pi_1\pi_{F_0} A\in \underline{F}_1&\text{ ($j$ even)}\\
S^1_1\pi_{F_0} A=\pi_1^\perp\pi_{F_0} A \in \underline{F}_1^\perp&\text{ ($j$ odd)}\\
S^1_0\pi_{F_0}^\perp A=\pi_1\pi_{F_0}^\perp A\in \underline{F}_1^\perp&\text{ ($j$ even)}\\
S^1_1\pi_{F_0}^\perp A=\pi_1^\perp\pi_{F_0}^\perp A\in \underline{F}_1&\text{ ($j$ odd)}\\
\end{array}\]

Let us now establish the induction: assume the result holds up to $r$ and that $\underline{F}_r$ splits $\underline{\alpha}_{r+1}$. Then,
\begin{equation}\label{Equation:FirstEquationINTheProofOfLemma:ProjectionsOfSijGeneralCase}
S^{r+1}_j\pi_{F_0} A=\pi_{r+1}S^r_j\pi_{F_0} A+\pi_{r+1}^\perp S^r_{j-1}\pi_{F_0} A.
\end{equation}
If $j$ is odd, $S^r_j\pi_{F_0} A\in \underline{F}_r^\perp$. Since $j-1$ is even, $S^r_{j-1}\pi_{F_0} A\in \underline{F}_r$. Hence, \eqref{Equation:FirstEquationINTheProofOfLemma:ProjectionsOfSijGeneralCase} becomes $\pi_{r+1}\pi_{F_r}^\perp S^r_j \pi_{F_0}A+\pi_{r+1}^\perp\pi_{F_r} S^r_{j-1}\pi_{F_0} A \in \underline{F}_{r+1}$. The remaining cases have similar proofs.

\end{proof}

We know that a harmonic map is obtained as the product of unitons $\underline{\alpha}_i$ with $\underline{\alpha}_i$ given as in Theorem \ref{Theorem:Theorem11FromFerSimWood09}. To obtain maps into $G_*(\mathbb C^n)$, we must impose the following algebraic conditions on the meromorphic data $H_{i,j}$:

\begin{proposition}\label{Proposition:GivingTheRightHsInTheGeneralCase}
Let $(H_{i,j})_{0\leq i\leq r-1, 1\leq j\leq n}$ be chosen in such a way that for all $j$, either $\pi_{F_0}(H_{0,j})=0$ and
\begin{equation}\label{Equation:EquationInProposition:GivingTheRightHsInTheGeneralCase}
\left\{
\begin{array}{l}
\pi_{F_0}(\displaystyle{\sum_{s=1}^i\binom{i-1}{s-1}H_{s,j})}=0,\text{ $i$ even},\\
\pi_{F_0}^\perp(\displaystyle{\sum_{s=1}^i\binom{i-1}{s-1}H_{s,j})}=0,\text{ $i$ odd}.
\end{array}
\right.
\end{equation}
or $\pi_{F_0}^\perp(H_{0,j})=0$ and \eqref{Equation:EquationInProposition:GivingTheRightHsInTheGeneralCase} holds, now with $\pi_{F_0}$ replaced with $\pi_{F_0}^\perp$.

Then $\phi=(\pi_{F_0}-\pi_{F_0}^\perp)(\pi_1-\pi_1^\perp)...(\pi_{r}-\pi_{r}^\perp)$, with $\underline{\alpha}_i$ given as in \eqref{Equation:FirstEquationInTheorem:Theorem11FromFerSimWood09}, lies in $G_*(\fieldC^n)$.

Moreover, for each $0\leq i\leq r-1$, $\phi_{i}=\pi_{F_{i}}-\pi_{F_{i}}^\perp$ with:

(i) $\underline{F}_{i}\cap\underline{\alpha}_{i+1}\subseteq\underline{F}_{i+1}$ spanned by $\underline{\alpha}_{i+1,j}^{(k)}$, where $j$ and $k$ are such that $k$ is even and $\pi_{F_0}^\perp(H_{0,j})=0$ or $k$ is odd and $\pi_{F_0}(H_{0,j})=0$;

(ii) $\underline{F}_{i}^\perp\cap\underline{\alpha}_{i+1}\subseteq \underline{F}_{i+1}^\perp$ spanned by $\underline{\alpha}_{i+1,j}^{(k)}$, where $j$ and $k$ are such that $k$ is odd and $\pi_{F_0}^\perp(H_{0,j})=0$ or $k$ is even and $\pi_{F_0}(H_{0,j})=0$.
\end{proposition}

\begin{proof}
For $r=1$, it is trivial. For the case $r=2$, our initial data satisfies, for each $j$, $\pi_{F_0}(H_{0,j})=0$ and $\pi_{F_0}^\perp(H_{1,j})=0$ or $\pi_{F_0}^\perp(H_{0,j})=0$ and $\pi_{F_0}(H_{1,j})=0$. Moreover, $\underline{\alpha}_2$ is spanned by $\alpha_{2,j}^{(0)}$ and $A_z^{\phi_1}(\alpha_{2,j}^{(0)})$. Now, $\alpha_{2,j}^{(0)}$ is either of the form $H_{0,j}+\pi_1^\perp H_{1,j}$ with $H_{0,j}$ in $F_0$ and $H_{1,j}$ in $F_0^\perp$ or $H_{0,j}+\pi_1^\perp H_{1,j}$ with $H_{0,j}$ in $F_0^\perp$ and $H_{1,j}$ in $F_0$. In the first case, $\alpha_{2,j}^{(0)}$ is a section of $\underline{F}_1$ (and of $\underline{\alpha}_2$ so that it is a section of $\underline{F}_2$) whereas in the second case we have a section of $\underline{F}_1^\perp$ (and hence of $\underline{F}_1^\perp\cap\underline{\alpha}_2\subseteq\underline{F}_2^\perp$). Since $A_z^{\phi_1}$ interchanges $\underline{F}_1$ with $\underline{F}_1^\perp$, we conclude that $\underline{F}_1\cap\underline{\alpha}_2$ is spanned by $\alpha_{2,j}^{(0)}$, for $j$ such that $\pi_{F_0}^\perp(H_{0,j})=0$, and by $\alpha_{2,j}^{(1)}=-A_z^{\phi_1}(\alpha_{2,j}^{(0)})$, for $j$ such that $\pi_{F_0}(H_{0,j})=0$.

Let us show the induction step: assume the result holds up to $r$. Without loss of generality, assume that $j$ is such that $\pi_{F_0}^\perp H_{0,j}=0$. Then, $\alpha_{r,j}^{(0)}$ lies in $\underline{F}_r$ and
\[\begin{array}{ll}\alpha_{r+1,j}^{(0)}&=\displaystyle{\alpha_{r,j}^{(0)}+\pi_r^\perp\big(\sum_{t=0}^{r-1}C_t^{r-1}H_{t+1,j}\big)}\\
~&=\displaystyle{\alpha_{r,j}^{(0)}+\pi_r^\perp\big(\sum_{t=0}^{r-1}S_{t}^{r-1}(\pi_{F_0}+\pi_{F_0}^\perp)\sum_{s=0}^t\binom{t}{s}H_{s+1,j}\big)}\\
~&=\displaystyle{\alpha_{r,j}^{(0)}+\pi_r^\perp\big(
\sum_{\text{\renewcommand{\arraystretch}{0.5}\tiny{$\begin{array}{l}t=0\\t\text{~odd}\end{array}$}\renewcommand{\arraystretch}{1}}}^{r-1}S_{t}^{r-1}\pi_{F_0}\sum_{s=1}^{t+1}\binom{t}{s-1}H_{s,j}+
\sum_{\text{\renewcommand{\arraystretch}{0.5}\tiny{$\begin{array}{l}t=0\\t\text{~even}\end{array}$}\renewcommand{\arraystretch}{1}}}^{r-1}S_{t}^{r-1}\pi_{F_0}^\perp\sum_{s=1}^{t+1}\binom{t}{s-1}H_{s,j}\big)}\\
~&+\displaystyle{\pi_r^\perp\big(\sum_{\text{\renewcommand{\arraystretch}{0.5}\tiny{$\begin{array}{l}t=0\\t\text{~odd}\end{array}$}\renewcommand{\arraystretch}{1}}}^{r-1}S_{t}^{r-1}\pi_{F_0}^\perp\sum_{s=1}^{t+1}\binom{t}{s-1}H_{s,j}+
\sum_{\text{\renewcommand{\arraystretch}{0.5}\tiny{$\begin{array}{l}t=0\\t\text{~even}\end{array}$}\renewcommand{\arraystretch}{1}}}^{r-1}S_{t}^{r-1}\pi_{F_0}\sum_{s=1}^{t+1}\binom{t}{s-1}H_{s,j}\big)}.
\end{array}\]
Using Lemma \ref{Lemma:ProjectionsOfSijGeneralCase}, the first two terms lie in $\underline{F}_r$ whereas the last vanishes from our hypothesis. Hence $\alpha_{r+1,j}^{(0)}\in \underline{F}_r\cap\underline{\alpha}_{r+1}\subseteq \underline{F}_{r+1}$. Since $\alpha_{r+1,j}^{(k)}=-A_z^{\phi_r}(\alpha_{r+1,j}^{(k-1)})$ and $A_z^{\phi_r}$ interchanges $\underline{F}_r$ and $\underline{F}_r^\perp$ the conclusion now easily follows.
\end{proof}

\textit{Proof of Theorem \ref{Theorem:MainTheoremNonBundleVersion}.}
Let $F_0$ be a constant subspace of $\fieldC^n$ and $(K_{i,j})_{0\leq i\leq r-1,1\leq j\leq n}$ denote a $F_0$-array. Let $H_{i,j}$ be defined as in \eqref{Equation:FirstEquationInTheorem:MainTheoremNonBundleVersion}. It is easily seen that these equations are equivalent to
\[\begin{array}{l}
H_{0,j}=K_{0,j}\text{ and}\\
K_{i,j}=\displaystyle{\sum_{s=1}^i\binom{i-1}{s-1}H_{s,j}.}
\end{array}
\]
From Proposition \ref{Proposition:GivingTheRightHsInTheGeneralCase}, we conclude that if $\phi$ is given as in Theorem \ref{Theorem:MainTheoremNonBundleVersion}, $\phi$ is harmonic and has values in $G_*(\fieldC^n)$. It remains to prove the converse, that all harmonic maps into $G_*(\fieldC^n)$ with finite uniton number can be given this way.

If $\phi$ has uniton number $1$, $\phi=(\pi_{F_0}-\pi_{F_0}^\perp)(\pi_1-\pi_1^\perp)$. Hence, $\phi$ lies in $G_*(\fieldC^n)$ if and only if $F_0$ splits $\underline{\alpha}_1$. But $\underline{\alpha}_1$ is spanned by some collection of $H_{i,j}$. Hence, it must be that we can choose the spanning set taking values either in $F_0$ or in $F_0^\perp$. In that case, $\phi=\pi_{F_1}-\pi_{F_1}^\perp$. If $\phi$ has uniton number $2$, then $\phi=(\pi_{F_1}-\pi_{F_1}^\perp)(\pi_2-\pi_2^\perp)$. Hence, $\phi$ takes values in $G_*(\fieldC^n)$ if and only if $\underline{F}_1$ splits $\underline{\alpha}_2$. But $\underline{\alpha}_2$ is spanned by vectors of the form $H_{0,j}+\pi_1^\perp H_{1,j}$ (and $A_z^{\phi_1}(H_{0,j}+\pi_1^\perp H_{1,j})$). Since $\underline{F}_1$ splits $\underline{\alpha}_2$, we must have $\pi_{F_1}(\underline{\alpha}_2)$ and $\pi_{F_1}^\perp(\underline{\alpha}_2)$ lying in $\underline{\alpha}_2$. Now, if $H_{0,j}$ lies in $F_0$ ($\pi_{F_0}^\perp H_{0,j}=0$), then it lies in $F_0\cap\underline{\alpha}_1\subseteq\underline{F}_1$. Hence,
\[\pi_{F_1}^\perp(H_{0,j}+\pi_1^\perp H_{1,j})=\pi_{F_1}^\perp(H_{0,j}+\pi_1^\perp\pi_{F_0}^\perp H_{1,j}+\pi_1^\perp\pi_{F_0}H_{1,j})=\pi_1^\perp\pi_{F_0} H_{1,j}\]
lies in $\underline{\alpha}_2$. Write $\tilde{H_1}=\pi_{F_0}(H_{1,j})$. Then, $\underline{\alpha}_2$ is spanned by $\pi_1^\perp\tilde{H}_{i,j}$ and $H_{0,j}+\pi_1^\perp \hat{H}_{1,j}$, where $\hat{H}_{1,j}=H_{1,j}-\tilde{H}_{1,j}$ lies in $F_0^\perp$.

In general, assume that $\pi_{F_0}^\perp H_{0,j}=0$ and $r$ is odd (the remaining cases are similar). Write
\[\begin{array}{ll}
\alpha_{r+1,j}^{(0)}&=\displaystyle{\alpha_{r,j}^{(0)}+\pi_r^\perp\Big(\sum_{t=0}^{r-1} C_t^{r-1}H_{t+1,j}\Big)}\\
~&=\displaystyle{\alpha_{r,j}^{(0)}+\pi_r^\perp\Big(\sum_{t=0}^{r-1}S_{t}^{r-1}(\pi_{F_0}+\pi_{F_0}^\perp)\sum_{s=0}^t\binom{t}{s}H_{s+1,j}\Big)}.\end{array}\]
By the induction hypothesis, $\alpha_{r,j}^{(0)}$ lies in $\underline{F}_{r}$. By Lemma \ref{Lemma:ProjectionsOfSijGeneralCase}, if $t$ is even,
\[\pi_r^\perp\Big(S_{t}^{r-1}\pi_{F_0}\sum_{s=0}^t\binom{t}{s}H_{s+1,j}\Big)\]
lies in $\underline{F}_r^\perp$ and
\[\pi_r^\perp\Big(S_{t}^{r-1}\pi_{F_0}^\perp\sum_{s=0}^t\binom{t}{s}H_{s+1,j}\Big)\]
lies in $\underline{F}_r$. For $t$ odd, changing the roles of $F_0$ and $F_0^\perp$ we get the same conclusion.

Since $\underline{F}_r$ splits $\underline{\alpha}_{r+1}$, we must have $\pi_{F_r}^\perp(\alpha_{r+1,j}^{(0)})$ and $\pi_{F_r}(\alpha_{r+1,j}^{(0)})$ in $\underline{\alpha}_{r+1}$. But
\[\pi_{F_r}^\perp(\alpha_{r+1,j}^{(0)})=\pi_r^\perp\Big(\sum_{\text{\renewcommand{\arraystretch}{0.5}\tiny{$\begin{array}{l}t=0\\t\text{~even}\end{array}$}\renewcommand{\arraystretch}{1}}}^{r-1}S_{t}^{r-1}\pi_{F_0}\sum_{s=0}^t\binom{t}{s}H_{s+1,j}\Big)+\pi_r^\perp\Big(\sum_{\text{\renewcommand{\arraystretch}{0.5}\tiny{$\begin{array}{l}t=0\\t\text{~odd}\end{array}$}\renewcommand{\arraystretch}{1}}}^{r-1}S_{t}^{r-1}\pi_{F_0}^\perp\sum_{s=0}^t\binom{t}{s}H_{s+1,j}\Big)\]
lies in $\underline{\alpha}_{r+1}$. By the induction hypothesis,
\[\sum_{\text{\renewcommand{\arraystretch}{0.5}\tiny{$\begin{array}{l}t=0\\ t\text{ even}\end{array}$}\renewcommand{\arraystretch}{1}}}^{r-1}\pi_{F_0}\sum_{s=0}^t\binom{t}{s}H_{s+1,j}=S_{r-1}^{r-1}\pi_{F_0}\sum_{s=0}^{r-1}\binom{r-1}{s}H_{s+1,j}\]
and
\[\sum_{\text{\renewcommand{\arraystretch}{0.5}\tiny{$\begin{array}{l}t=0\\ t\text{ odd}\end{array}$}\renewcommand{\arraystretch}{1}}}^{r-1}\pi_{F_0}^{\perp}\sum_{s=0}^s\binom{t}{s}H_{s+1,j}
=\sum_{\renewcommand{\arraystretch}{0.5}\text{\tiny{$\begin{array}{l}t=0\\ t\text{ odd}\end{array}$}\renewcommand{\arraystretch}{1}}}^{r-2}\pi_{F_0}^{\perp}\sum_{s=0}^t\binom{t}{s}H_{s+1,j}=0.\]
Hence
\[\pi_r^\perp S_{r-1}^{r-1}\pi_{F_0}\sum_{s=0}^{r-1}\binom{r-1}{s}H_{s+1,j}=C_r^r\pi_{F_0}\sum_{s=0}^{r-1}\binom{r-1}{s}H_{s+1,j}\]
lies in $\underline{\alpha}_{r+1}$.

Take $\tilde{H}_j$ the holomorphic vector field given by
\[\tilde{H}_j=\pi_{F_0}\sum_{s=0}^{r-1}\binom{r-1}{s}H_{s+1,j}.\]
Then we can write
\[\begin{array}{ll}\pi_{F_r}(\alpha_{r+1,j}^{(0)})&=\alpha_{r+1,j}^{(0)}-\pi_{F_r}^\perp(\alpha_{r+1,j}^{(0)})\\
~&=\alpha_{r+1,j}^{(0)}-\pi_{r}^\perp\left(C^{r-1}_0H_{1,j}+C^{r-1}_1H_{2,j}+...+C^{r-1}_{r-1}(H_{r,j}-\tilde{H}_{r,j})\right).\end{array}\]
Writing $\hat{H}_{r,j}=H_{r,j}-\tilde{H}_{r,j}$, we have
\[\wordspan\{\alpha_{r+1,j}^{(0)}\}=\wordspan\{C^r_r\tilde{H}_{r,j},C^r_0 H_{0,j}+...+C^{r}_{r-1}H_{r-1,j}+C^r_r\hat{H}_{r,j}\}.\]
We shall check that this new holomorphic data satisfies our conditions. As a matter of fact, $\pi_{F_0}^\perp(H_{0,j})=0$ and
\[\pi_{F_0}\Big(\sum_{s=1}^{r-1}\binom{r-1}{s-1}H_{s,j}+\hat{H}_{r,j}\Big)=\pi_{F_0}\Big(\sum_{s=1}^{r}\binom{r-1}{s-1}H_{s,j}\Big)-\tilde{H}_{r,j}=0.\]
Also, $\pi_{F_0}(0)=0$ and $\pi_{F_0}^\perp(\tilde{H}_{r,j})=0$, concluding our proof.
\qed

\bigskip

\textit{Proof of Proposition \ref{Proposition:DecompositionIntoHoloAndAntiHoloPartsOfHarmonicMapsIntoGStarCn}.}
We must show that $\phi\cap\underline{\alpha}$ and $\phi\cap\underline{\alpha}^\perp$ are, respectively, holomorphic and anti-holomorphic subbundles of $(\fieldC^n,D^\phi_{\bar{z}})$. From Proposition \ref{Proposition:GivingTheRightHsInTheGeneralCase} we know $D_{\bar{z}}^\phi$-holomorphic basis for $\phi\cap\underline{\alpha}$ and for $\phi\cap\underline{\alpha}^\perp$. Hence,
\[D_{\bar{z}}^\phi(\phi\cap\underline{\alpha})\subseteq\phi\cap\underline{\alpha}\text{ and }D_{\bar{z}}^\phi(\phi\cap\underline{\alpha}^\perp)\subseteq\phi\cap\underline{\alpha}^\perp.\]
Since $D_z^\phi\underline{\alpha}^\perp\subseteq\underline{\alpha}^\perp$, the result follows from the identity
\[<D_z^\phi(\phi^\perp\cap\underline{\alpha}^\perp),\phi\cap\underline{\alpha}^\perp>=<\phi^\perp\cap\underline{\alpha}^\perp,D_{\bar{z}}^\phi(\phi\cap\underline{\alpha}^\perp)>=0.\]

\qed

\bigskip
In order to prove Theorem \ref{Theorem:MainTheoremForUnitonCounting}, we start with the following Lemma.

\begin{lemma}\label{Lemma:LemmaToProveTheorem:MainTheoremForUnitonCounting}
Let $r \in \{1,...,n-1)$, $F_0$ be a $k$-dimensional subspace of $\mathbb C^n$ and consider a pair $(L,S)$ adapted to $F_0$. For any $F_0$-array $(K_{i,j})_{0\leq i\leq r-1, 1\leq j\leq n}$  which matches $(L,S)$,

(i) $\displaystyle{\wordrank(\underline{\alpha}_{i+1})=\wordrank(\underline{\alpha}_i)+\sum_{t=0}^{i}(l_t^{i-t}+s_t^{i-t})}$;

(ii) $\wordrank(\underline{\alpha}_{i+1}\cap\underline{F}_{i})=\left\{
\begin{array}{l}
\displaystyle{\wordrank(\underline{\alpha}_{i}\cap\underline{F}_{i-1})+\sum_{t=0}^{i}l_{t}^{i-t}\text{, if $i$ even}}\\
\displaystyle{\wordrank(\underline{\alpha}_{i}\cap\underline{F}_{i-1})+\sum_{t=0}^{i}s_{t}^{i-t}\text{, if $i$ odd}};
\end{array}
\right.$

(iii) $\wordrank(\underline{\alpha}_{i+1}\cap\underline{F}_i^\perp)=\left\{
\begin{array}{l}
\displaystyle{\wordrank(\underline{\alpha}_{i}\cap \underline{F}^{\perp}_{i-1})+\sum_{t=0}^{i}s_{t}^{i-t}\text{, if $i$ even}}
\\
\displaystyle{\wordrank(\underline{\alpha}_{i}\cap \underline{F}^{\perp}_{i-1})+\sum_{t=0}^{i}l_{t}^{i-t}\text{, if $i$ odd}}.
\end{array}\right.$

for each $i\in \{1,...,r-1\}$.
\end{lemma}

\bigskip
\begin{proof}
Let $i\in \left\{1,...,r-1\right\} $. We know that $\underline{\alpha} _{i+1}$ is spanned by $\{\underline{\alpha}_{i+1,j}^{(k)}\}_{0\leq k\leq i, 1\leq j\leq n}$. Since our array matches the pair $\left( L,S\right) $, we split  $\underline{\alpha}_{i+1}$, considering $\underline{\alpha}_{i+1}=P\oplus Q$, where

\[P=\wordspan\left\{\underline{\alpha}_{i+1,j_{k}}^{(k)}\right\}_{\text{\tiny{$
\begin{array}{l}
0\leq k\leq i-1 \\
1\leq j_{k}\leq A_{i-k}
\end{array}$}}}\text{ and }Q=\wordspan\left\{ C_{i}^{i}H_{i-k,j_{k}}^{(k)}\right\}_{\text{\tiny{$
\begin{array}{l}
0\leq k\leq i \\
A_{i-k}+1\leq j_{k}\leq A_{i-k+1}
\end{array}$}}
}.\]

The matching condition tells us that $K_{l,j}=0$ whenever $l<i$ and $A_i \leq j \leq A_{i+1}$. Then, for every
$0 \leq k \leq i$ and $A_{i-k}+1 \leq j_k \leq A_{i-k+1}$, we have
$C_{i}^{i}H_{i-k,j_{k}}^{(k)}=C_{i}^{i}K_{i-k,j_{k}}^{(k)}$; hence, from Definition \ref{Definition:ArrayMatchingPair},
$\wordrank(Q)=\sum_{j=0}^{i}l_{j}^{i-j}$, if $i$ even, and $\wordrank(Q)=\sum_{j=0}^{i}s_{j}^{i-j}$, if $i$ odd.

The matching condition ensures also that, whenever $0\leq k\leq i-1$ and $1\leq j\leq A_{i-k},$ $\alpha_{i,j}^{(k)}=\sum_{l=k}^{i-1}C_{l}^{i-1}H_{l-k,j}^{(k)}\neq 0$. Writing, as before, $\alpha_{i+1,j}^{(k)}=\alpha_{i,j}^{(k)}+\pi_{i}^{\perp}\left(\sum_{l=0}^{i-1}S_{l}^{i-1}\sum_{t=0}^{l}\binom{l}{t}H_{t+1,j}\right)$, we conclude that $\wordrank(P)=\wordrank(\alpha_{i,j})$, proving (i). On the other hand, since
$\underline{F}_{i}=\underline{\alpha}_{i}\cap \underline{F}_{i-1}\oplus\underline{\alpha}_i^\perp\cap\underline{F}_{i-1}^{\perp}$, we also obtain
$\wordrank(P\cap\underline{F}_i)=\wordrank(\underline{\alpha}_{i}\cap \underline{F}_{i-1})$, thus $\wordrank(\underline{\alpha}_{i+1}\cap\underline{F}_i)=\wordrank(\underline{\alpha}_i\cap\underline{F}_{i-1})+\wordrank(Q)$ getting (ii). The proof of (iii) is analogous, just interchanging $\underline{F}$ with $\underline{F}^{\perp}$ and $l$ with $s$.
\end{proof}

From the previous Lemma, using an induction argument, we easily get the following identities:

\bigskip

\begin{corollary}\label{Corollary:CorollaryToProveTheorem:MainTheoremForUnitonCounting}
Under the above conditions the following equalities hold:

(i) $\wordrank(\underline{\alpha}_{i+1}\cap \underline{F}_{i})=\left\{
\begin{array}{l}
\displaystyle{\sum_{j=0}^{\frac{i}{2}}\sum_{t=0}^{2j}l_{t}^{2j-t}+\sum_{j=0}^{\frac{i-2}{2}}\sum_{t=0}^{2j+1}s_{t}^{2j+1-t}\text{, if $i$ even}}
\\
\displaystyle{\sum_{j=0}^{\frac{i-1}{2}}\sum_{t=0}^{2j} l_{t}^{2j-t}+\sum_{j=0}^{\frac{i-1}{2}}\sum_{t=0}^{2j+1}s_{t}^{2j+1-t}\text{, if $i$ odd}};
\end{array}
\right.$

(ii) $\wordrank(\underline{\alpha}_{i+1}\cap\underline{F}^{\perp}_{i})=\left\{
\begin{array}{l}
\displaystyle{\sum_{j=0}^{\frac{i}{2}}\sum_{t=0}^{2j}s_{t}^{2j-t}+\sum_{j=0}^{\frac{i-2}{2}}\sum_{t=0}^{2j+1}l_{t}^{2j+1-t}\text{, if $i$ even}}\\
\displaystyle{\sum_{j=0}^{\frac{i-1}{2}}\sum_{t=0}^{2j}s_{t}^{2j-t}+\sum_{j=0}^{\frac{i-1}{2}}\sum_{t=0}^{2j+1}l_{t}^{2j+1-t}\text{, if $i$ odd}.}
\end{array}
\right. $
\end{corollary}

\bigskip

\textit{Proof of Theorem \ref{Theorem:MainTheoremForUnitonCounting}.}

From our data, we have, of course, $\wordrank(\underline{\alpha}_{1}\cap F_{0})=l_{0}^{0}$ and
$\wordrank(\underline{\alpha}_{1}\cap F_{0}^{\perp })=s_{0}^{0}$. Hence, since $\underline{F}_{1}=\underline{\alpha}_{1}\cap F_{0}\oplus\underline{\alpha}_{1}^\perp\cap F_{0}^{\perp}$, $\wordrank(\underline{F}_{1})=l_{0}^{0}+n-k-\wordrank(\underline{\alpha}_{1}\cap F_{0}^{\perp})=n-\left[k+s_{0}^{0}-l_{0}^{0}\right]$.

\bigskip

Analogously, from the equality $\underline{F}_{2}=\underline{\alpha}_{2}\cap\underline{F}_{1}\oplus\underline{\alpha}_{2}^{\perp}\cap\underline{F}_{1}^{\perp}$, we get, using Lemma \ref{Lemma:LemmaToProveTheorem:MainTheoremForUnitonCounting} and Corollary \ref{Corollary:CorollaryToProveTheorem:MainTheoremForUnitonCounting}
\[\begin{array}{ll}
\wordrank(\underline{F}_{2})&=l_{0}^{0}+s_{0}^{1}+s_{1}^{0}+\wordrank(\underline{F}_{1}^{\perp})-\wordrank(\underline{\alpha}_{2}\cap\underline{F}_{1}^{\perp})\\
~&=l_{0}^{0}+s_{0}^{1}+s_{1}^{0}+k+s_{0}^{0}-l_{0}^{0}-(s_{0}^{0}+l_{0}^{1}+l_{1}^{0})\\
~&=k+\displaystyle{\sum_{j=0}^{1}(s_{j}^{2t+1-j}-l_{j}^{2t+1-j})}.\end{array}\]

Assume now that the proposition holds for $\underline{F}_{i}$, where $i$ is even (the proof for $i$ odd is analogous).

The equality $\underline{F}_{i+1}=\underline{\alpha} _{i+1}\cap\underline{F}_{i}\oplus\underline{\alpha}_{i+1}^{\perp}\cap\underline{F}_{i}^{\perp }$ implies that
\[\begin{array}{ll}\wordrank(\underline{F}_{i+1})&=\wordrank(\underline{\alpha}_{i+1}\cap\underline{F}_{i})+\wordrank(\underline{\alpha}_{i+1}^{\perp}\cap\underline{F}_{i}^{\perp})\\
~&=\wordrank(\underline{\alpha}_{i+1}\cap\underline{F}_{i})+\wordrank(\underline{F}_{i}^{\perp})-\wordrank(\underline{\alpha}_{i+1}\cap\underline{F}_{i}^{\perp}).\end{array}\]

\bigskip

Now, using Lemma \ref{Lemma:LemmaToProveTheorem:MainTheoremForUnitonCounting}, we get \[\wordrank(\underline{F}_{i+1})=\wordrank(\underline{\alpha}_{i}\cap\underline{F}_{i-1})+\sum_{t=0}^{i}l_{t}^{i-t}+n-\wordrank(\underline{F}_{i})-\wordrank(\underline{\alpha}_{i}\cap\underline{F}_{i-1}^{\perp})-\sum_{t=0}^{i}s_{t}^{i-t}.\]

\bigskip

From Corollary \ref{Corollary:CorollaryToProveTheorem:MainTheoremForUnitonCounting} and the knowledge of $\wordrank(\underline{F}_{i})$, we conclude that
\[\begin{array}{ll}
\wordrank(\underline{F}_{i+1})&=\displaystyle{n-k-\sum_{j=0}^{\frac{i}{2}-1}\sum_{t=0}^{2j+1}(s_{t}^{2j+1-t}-l_{t}^{2j+1-t})+\sum_{j=0}^{\frac{i-2}{2}}\sum_{t=0}^{2j}(l_{t}^{2j-t}-s_{t}^{2j-t})}\\
~&+\displaystyle{\sum_{j=0}^{\frac{i-2}{2}}\sum_{t=0}^{2j+1}(s_{t}^{2j+1-t}-l_{t}^{2j+1-t})+\sum_{j=0}^{i}(l_{j}^{i-j}-s_{j}^{i-j})}\\
~&=\displaystyle{n-\Big[k+\sum_{j=0}^{\frac{i}{2}-1}\sum_{t=0}^{2j}(s_{t}^{2j-t}-l_{t}^{2j-t})\Big]+\sum_{j=0}^{i}(l_{j}^{i-j}-s_{j}^{i-j})}\\
~&=\displaystyle{n-\Big[ k+\sum_{j=0}^{\frac{i}{2}}\sum_{t=0}^{2j}(s_{j}^{2t-j}-l_{j}^{2t-j})\Big]},\end{array}\]
as wanted.
\qed

\textit{Proof of Theorem \ref{Theorem:EstimatesOnTheUnitonNumber}.}

Given a matrix $D$, we will let $D_i$ denote its $i$'th column. We consider $k \in\left \{1,...,n \right\}$ fixed and let $r_k$ denote the maximal uniton number for harmonic maps $\varphi=(\pi_0-\pi_0^{\perp})(\pi_1-\pi_1^{\perp})...(\pi_i-\pi_i^{\perp}) : S^2 \rightarrow G_*(\mathbb C^n)$, where $F_0$ is a complex subspace of $\mathbb C^n$ with dimension $k$.

We will first show that it is not possible to have simultaneously $r_k \geq k$ and $r_k \geq n-k$. Indeed, these two conditions would imply the existence of a pair $(L,S)$ of $r_k\times r_k$ matrices, adapted to $F_0$, matching a given array; from the fullness of $\underline{\alpha}_1$ we would get\renewcommand{\arraycolsep}{5pt}
\[L_{1}=
\begin{array}{rl}
\left[
\begin{array}{c}
1 \\
\vdots  \\
1 \\
0 \\
\vdots \\
0
\end{array}
\right]&
\hspace{-5mm}\begin{array}{c}
\left\}\text{$\begin{array}{l}~\\~\\~\end{array}$}\right.\hspace{-5mm}n-k\\
~\\
~\\
~\\
\end{array}
\end{array}
\text{ and }
S_{1}=
\begin{array}{rl}
\left[
\begin{array}{c}
1 \\
\vdots  \\
1 \\
0 \\
\vdots \\
0
\end{array}
\right]&
\hspace{-5mm}\begin{array}{c}
\left\}\text{$\begin{array}{l}~\\~\\~\end{array}$}\right.\hspace{-5mm}k\\
~\\
~\\
~\\
\end{array}
\end{array}
\]\renewcommand{\arraycolsep}{1pt}\noindent
which cannot happen since the sum of the entries of both matrices must be strictly less than $n$.

We have to analyze, separately, the different situations $k<p$ and $k\geq p$. The techniques are similar, so that we only present the first case.

Consider that $k<p$ and let $(L,S)$ be a pair of $r_k \times r_k$ matrices, adapted to $F_0$ and matching an $F_0$-array. It is easily seen that $k<2p-k\leq n-k$. Consider $k<r_k\leq n-k$. Assume that $r_k$ is even. From Theorem \ref{Theorem:MainTheoremForUnitonCounting}, we can write
\[p-k = \sum_{j=0}^{\frac{r_k}{2}-1}\sum_{t=0}^{2j+1}(s_t^{2j+1-t}-l_t^{2j+1-t}).\]
The fullness of $\underline{\alpha}_1$ implies\renewcommand{\arraycolsep}{5pt}
\[L_{1}=
\begin{array}{rl}
\left[
\begin{array}{c}
1 \\
\vdots  \\
1 \\
0 \\
\vdots \\
0
\end{array}
\right]&
\hspace{-5mm}\begin{array}{c}
\left\}\text{$\begin{array}{l}~\\~\\~\end{array}$}\right.\hspace{-5mm}k\\
~\\
~\\
~\\
\end{array}
\end{array},\,L_i=0,\text{ if }i>1\text{ and }S_{1}=\left[
\begin{array}{c}
1 \\
\vdots\\
1
\end{array}
\right],\]\renewcommand{\arraycolsep}{1pt}\noindent
so that

\[\sum_{j=0}^{\frac{r_k}{2}-1}(s_0^{2j+1}-l_0^{2j+1})=\frac{r_k-k+a_k}{2}.\]
Hence, $p-k=\frac{r_k-k+a_k}{2} + \theta$, where $0 \leq \theta \leq n-r_k-1$. Therefore $p-k \geq \frac{r_k-k+a_k}{2}$, which implies $r_k \leq 2p-k-a_k$.

If $r_k$ is odd we will get instead
\begin{equation*}
\begin{array}{ll}
n-(p+k)&=\displaystyle{\sum_{j=0}^{\frac{r_k-1}{2}}\sum_{t=0}^{2j}(s_t^{2j-t}-l_t^{2j-t})}\\
~&=\displaystyle{\frac{r_k-k+1-a_k}{2}+\sum_{j=0}^{\frac{r_k-1}{2}}\sum_{t=1}^{2j}(s_t^{2j}-l_t^{2j})<\frac{r_k-k+1-a_k}{2}+n-k-r_k.}
\end{array}
\end{equation*}
Hence $r_k<2p-k-a_k+1$, or $r_k \leq 2p-k-a_k$.

These estimates are sharp as we may easily see. For instance, in the case $r_k$ odd, we can consider the pair $(L,S)$ of order $2p-k-a_k$ with\renewcommand{\arraycolsep}{5pt}
\[L_{1}=
\begin{array}{rl}
\left[
\begin{array}{c}
1 \\
\vdots  \\
1 \\
0 \\
\vdots \\
0
\end{array}
\right]&
\hspace{-5mm}\begin{array}{c}
\left\}\text{$\begin{array}{l}~\\~\\~\end{array}$}\right.\hspace{-5mm}k\\
~\\
~\\
~\\
\end{array}
\end{array},\,L_i=S_i=0\text{ if }i>1\text{ and }S_1=\left[
\begin{array}{c}
1 \\
\vdots \\
1
\end{array}
\right].\]\renewcommand{\arraycolsep}{1pt}\noindent
Taking meromorphic functions $L_{0,1}$ and $E_{0,1}$ with such that
\[\wordspan\{L_{0,1},L_{0,1}^{(1)},...,L_{0,1}^{(k)}\}=F_0\text{ and }\wordspan\{E_{0,1},E_{0,1}^{(1)},...,E_{0,1}^{(n-k)}\}=F_0^\perp,\]
we get an array matching $(L,S)$.

Then $\sum_{j=0}^{\frac{r_k-1}{2}}\sum_{t=0}^{2j}(s_t^{2j-t}-l_t^{2j-t})=\frac{2p-k-a_k-k+a_k}{2}=p-k$, concluding the proof.




\begin{thebibliography}{99}

\bibitem{BurstallGuest:97} F.~E.\ Burstall and M.~A.\ Guest, \textit{Harmonic two-spheres in compact symmetric spaces, revisited}, Math. Ann. 309 (1997) 541--572.

\bibitem{BurstallWood:86} F.~E.\ Burstall and J.~C.\ Wood, \emph{The construction of harmonic maps into complex Grassmannians}, J. Diff. Geom. {\bf 23} (1986), 255--298.

\bibitem{FerreiraSimoesWood:09} M.~J. Ferreira, B.~A Simões and J.~C. Wood \emph{All harmonic $2$-spheres in the unitary group, completely explicitly}, Math. Z. (print version to apeear).

\bibitem{ChernWolfson:85} S.~S. Chern and J.~G. Wolfson \emph{ Harmonic maps of $S^2$ into a complex Grassmann manifold}, J. Proc. Nat. Acad. Sci. 82 (1985), 2217--2219.

\bibitem{ChernWolfson:87} S.~S. Chern and J.~G. Wolfson \emph{Harmonic maps of the two-sphere into a complex Grassmann manifold II}, J. Ann. of Math. 125 (1987), 301--335.

\bibitem{DaiTerng:07} B.~Dai and C.~-L. Terng, \emph{Bäcklund transformations, Ward solitons, and unitons}, J. Differential Geom. 75 (2007), 57--108.

\bibitem{DongShen:96} Y.~Dong and Y.~Shen, \emph{Factorization and uniton numbers for harmonic maps into the unitary group $\U(n)$}, Sci. China Ser. A 39 (1996), 589--597.

\bibitem{Guest:97} M.~A.\ Guest, \textit{Harmonic maps, loop groups, and integrable systems}, London Mathematical Society Student Texts, 38, Cambridge University Press, Cambridge, 1997.

\bibitem{Guest:02} M.~A. Guest, \emph{An update on harmonic maps of finite uniton number, via the zero curvature equation}, Integrable systems, topology, and physics (Tokyo, 2000),  85--113, Contemp.\ Math. \textbf{309}, Amer. Math. Soc., Providence, RI, 2002.

\bibitem{HeShen:04} Q. He and Y.~B. Shen, \emph{Explicit construction for harmonic surfaces in $\U(n)$ via adding unitons}, Chinese Ann. Math. Ser. B  \textbf{25}  (2004) 119--128.

\bibitem{KoszulMalgrange:58} J.~L. Koszul and B.~Malgrange, \emph{Sur certaines structures fibrées complexes}, Arch. Math. {\bf 9} (1958), 102--109.

\bibitem{PressleySegal:86} A.\ Pressley and G.\ Segal, \textit{Loop groups}, Oxford Mathematical Monographs, Oxford Science Publications, The Clarendon Press, Oxford University Press, Oxford, 1986.

\bibitem{Segal:89} G.~Segal, \emph{Loop groups and harmonic maps}, Advances in homotopy theory (Cortona, 1988), 153--164, London Math. Soc. Lecture Note Ser., 139, Cambridge Univ. Press, Cambridge, 1989.


\bibitem{Uhlenbeck:89} K.~Uhlenbeck, \emph{Harmonic maps into Lie groups: classical solutions of the chiral model}, J. Differential Geom. \textbf{30} (1989), 1--50.

\bibitem{Wood:89} J.~C. Wood, \emph{Explicit construction and parametrization of harmonic two-spheres in the unitary group}, Proc. London Math. Soc. (3) \textbf{58} (1989) 608--624.

\end{thebibliography}
\end{document}